\documentclass[12pt,tbtags,leqno]{amsart}

\usepackage{graphicx}
\usepackage{yfonts}

\usepackage[T1]{fontenc}
\usepackage[latin1]{inputenc}
\usepackage[english]{babel}
\usepackage{amssymb}
\usepackage{amsthm}
\usepackage{amsmath}
\usepackage{verbatim}
\usepackage{color}

\numberwithin{equation}{section}

\newcommand{\MH}{\mathrm{Mult}(\mathcal{H})}

\newcommand{\HH}{\mathcal{H}}
\newcommand{\HK}{\mathcal{K}}
\newcommand{\HM}{\mathcal{M}}

\newcommand{\HC}{\mathcal{C}}

\newcommand{\dB}{\partial \mathbb B_d}

\newcommand{\D}{\mathbb{D}}
\newcommand{\B}{\mathbb{B}}
\newcommand{\C}{\mathbb{C}}
\newcommand{\N}{\mathbb{N}}
\newcommand{\R}{\mathbb{R}}

\newcommand{\binoN}{\left( \begin{matrix} N\\ \alpha \end{matrix}\right)}

\newcommand{\la}{\langle}
\newcommand{\ra}{\rangle}
\newcommand{\Bd}{\mathbb B_d}

\DeclareMathOperator{\dist}{dist}
\newcommand{\Hol}{\operatorname{Hol}}
\newcommand{\Mult}{\mathrm{Mult}}
\newcommand{\clos}{\operatorname{clos}}

\def\om{\omega}

\def\Bd{{\mathbb{B}_d}}

\newcommand{\Ker}[1]{\mathsf{Ker}~}

\DeclareMathOperator{\stable}{stable}

\theoremstyle{plain}
\newtheorem{theorem}{Theorem}[section]
\newtheorem{lemma}[theorem]{Lemma}

\newtheorem{prop}[theorem]{Proposition}
\newtheorem{corollary}[theorem]{Corollary}

\newtheorem{ques}[theorem]{Question}
\theoremstyle{definition}
\newtheorem{example}[theorem]{Example}

\begin{document}

\date{\today}

\bibliographystyle{plain}
\title[Cyclicity in the Drury-Arveson space]{Cyclicity in the Drury-Arveson space and other weighted Besov spaces}
\author[A. Aleman]{Alexandru Aleman}
\address{Lund University, Mathematics, Faculty of Science, P.O. Box 118, S-221 00 Lund, Sweden}
\email{alexandru.aleman@math.lu.se}

\author[K.-M. Perfekt]{Karl-Mikael Perfekt}
\address{Department of Mathematical Sciences, Norwegian University of Science and Technology (NTNU), 7491 Trondheim, Norway}
\email{karl-mikael.perfekt@ntnu.no}

\author[S. Richter]{Stefan Richter}
\address{Department of Mathematics, University of Tennessee, 1403 Circle Drive, Knoxville, TN 37996-1320, USA}
\email{srichter@utk.edu}

\author[C. Sundberg]{Carl Sundberg}
\address{Department of Mathematics, University of Tennessee, 1403 Circle Drive, Knoxville, TN 37996-1320, USA}
\email{csundber@utk.edu}

\author[J. Sunkes]{James Sunkes}
\address{Huntsville, AL, USA}
\email{jsunkes@gmail.com}


\begin{abstract} Let $\HH$ be a space of analytic functions on the unit ball  $\Bd$  in $\C^d$ with multiplier algebra $\mathrm{Mult}(\HH)$. A function $f\in \HH$ is called cyclic if the set $[f]$, the closure of $\{\varphi f: \varphi \in \mathrm{Mult}(\HH)\}$, equals $\HH$. For multipliers we also consider a weakened form of the cyclicity concept. Namely  for $n\in \N_0$ we consider the classes $$\HC_n(\HH)=\{\varphi \in \MH:\varphi\ne 0, [\varphi^n]=[\varphi^{n+1}]\}.$$  Many of our results hold for $N$:th order radially weighted Besov spaces on $\Bd$, $\HH= B^N_\om$, but we describe our results only for the  Drury-Arveson space $H^2_d$ here.
	
Letting $\C_{\stable}[z]$ denote the stable polynomials for $\Bd$, i.e. the $d$-variable complex polynomials without zeros in $\Bd$, we show that
\begin{align*} &\text{ if } d \text{ is odd, then } \C_{\stable}[z]\subseteq \HC_{\frac{d-1}{2}}(H^2_d), \text{ and }\\
&\text{ if } d \text{ is even, then } \C_{\stable}[z]\subseteq \HC_{\frac{d}{2}-1}(H^2_d).\end{align*}
For $d=2$ and $d=4$ these inclusions are the best possible, but in general we can only show that if $0\le n\le \frac{d}{4}-1$, then $\C_{\stable}[z]\nsubseteq \HC_n(H^2_d)$.

For functions other than polynomials we show that if $f,g\in H^2_d$ such that $f/g\in H^\infty$ and $f$ is cyclic, then $g$ is cyclic. We use this to prove that if $f,g\in H^2_d$ extend to be analytic in a neighborhood of $\overline{\Bd}$, have no zeros in $\Bd$, and their zero sets coincide on the boundary, $Z(f)\cap \partial \Bd = Z(g)\cap \dB$, then $f$ is cyclic if and only if $g$ is cyclic. Furthermore, if  for  $f\in H^2_d\cap C(\overline{\Bd})$
the set $Z(f)\cap \dB$ embeds a cube of real dimension $\ge 3$, then $f$ is not cyclic in the Drury-Arveson space. On the other hand, stable polynomials with a finite zero set in $\partial \Bd$ are cyclic.
\end{abstract}

  \maketitle

\section{Introduction}
Investigations about cyclic vectors of spaces of single variable analytic functions can be considered to be classical. Beurling \cite{Beurling} showed that the cyclic vectors of the Hardy space $H^2$ are the outer functions, that is, those $f\in H^2$ such that
$$-\infty < \log|f(0)|= \int_0^{2\pi} \log|f(e^{it})| \frac{dt}{2\pi}.$$ Korenblum \cite{KorenblumH2n, KorenblumBeurling} established complete results for the topological algebra $A^{-\infty}=\{f\in \Hol(\D): |f(z)|=O((1-|z|)^{-n}) \text{ for some }n\in \N\}$, and for the weighted Dirichlet spaces $\{f\in \Hol(\D): f^{(n)}\in H^2\}$. Complete characterizations of the cyclic functions in the Bergman space $L^2_a=\{f\in \Hol(\D): \int_\D |f|^2dA<\infty\}$ and Dirichlet space $D=\{f\in \Hol(\D): f'\in L^2_a\}$  are lacking, but the area is rich with deep results that have clarified the function theory for these spaces, see for example \cite{DurenSchuster, ElFallahKellayMashreghiRansfordPrimer, HedKorZhu}.

Much less is known in the setting of spaces of functions of several complex variables. Norm closed ideals in the ball and polydisc algebras have been investigated by Hedenmalm, \cite{HedenmalmBall, HedenmalmBidisc}, and there are results for the analogous questions in the Drury-Arveson context by Clou\^{a}tre and Davidson, \cite{ClDavidsonIdeals}. Cyclicity of polynomials in weighted Dirichlet spaces of the bidisc has been investigated by B\'en\'eteau-Condori-Liaw-Seco-Sola \cite{BenetCondoriEtAl}, B\'en\'eteau-Knese-Kosi\'nski-Liaw-Seco-Sola \cite{BenetKneseKosLiaw}, and Knese-Kosi\'nski-Ransford-Sola \cite{KneseKosinskiRansfordSola}. The paper \cite{GuoZhou} by Guo-Zhou contains results specifically for the Hardy space of the bidisc, while the paper \cite{LinusBergq} by Bergqvist treats the polydisc. In \cite{AlanSola} Sola extended known bidisc results to the unit ball of two complex variables. Some cyclicity results for the Drury-Arveson space $H^2_d$ can be found in \cite[Theorems~1.4 and 1.5]{RiSunkesHankel}. We refer the reader to \cite{HartzDruryArv} for a general introduction to the Drury-Arveson space $H^2_d$.

The purpose of this paper is to study the properties of both cyclic and non-cyclic functions $f\in H^2_d$ which have no zeros in $\Bd$.  Since $H^2_d$ is central to multivariable operator theory connected with $\Bd$, we expect our results to have significance in that context.  The Drury-Arveson space is known to be an example of two types of important classes of function spaces: it is a Hilbert function space with a complete Pick kernel, and it is a radially weighted Besov space on the unit ball of $\C^d$. Many of our methods apply in this generality. Thus, our presentation will be very general. In the special case of the Dirichlet space of the unit disc our approach provides a new proof of Corollary 5.5 of \cite{RiSuMMJ}.

 It is known that for complete Pick spaces there is a 1-1 correspondence between multiplier invariant subspaces of $\HH$ and weak* closed ideals of $\MH$, see \cite{DavidsonRamseyShalit}. As our main interest is in $H^2_d$ we have stated our results as results about invariant subspaces of $B^N_\om$, even though many of them are also results about weak* closed ideals in $\Mult(B^N_\om)$.

Before stating and discussing our main results in Section~\ref{sec:main}, we will present necessary background material on radially weighted Besov spaces and their multipliers in Section~\ref{sec:prelim}. In this preliminary section we also introduce the classes $\HC_n(\HH)$ of pseudo-cyclic multipliers. In Section~\ref{sec:ex} we illustrate our results through a number of examples. In Section~\ref{sec:besov} we prove the theorems which apply for general radially weighted Besov spaces, while we treat those spaces which are also complete Pick spaces in Section~\ref{SecPickSpaces}. Finally, we consider some natural open questions in Section~\ref{sec:questions}.

\section{Preliminaries} \label{sec:prelim}
Let $X$ be a set. A Hilbert function space $\HH$ on $X$ is a Hilbert space of complex valued functions on $X$ such that for each $z\in X$ the evaluation functional $f\to f(z)$ is continuous on $\HH$.

Let $d\in \N$ and let $\Bd$  denote the open unit ball of $\C^d$. We will use $\Hol(\Bd)$ to denote the analytic functions on $\Bd$, and we will write $\D=\mathbb{B}_1$, when we want to emphasize that $d=1$. The main focus of this paper will be on radially weighted Besov spaces $$B^N_\om=\{f\in \Hol(\Bd): R^Nf\in L^2(\om)\}.$$ Here $N$ is a non-negative integer, $R=\sum_{j=1}^d z_j \frac{\partial}{\partial z_j}$ is the radial derivative operator, and $\om$ is an \textit{admissible radial measure} on $\overline{\Bd}$. That is, $\om$ is of the type $d\om(z)=d\mu(r) d\sigma(w)$, where $z=rw$, $\sigma$ is the normalized rotationally invariant measure on $\dB$, and $\mu$ is a Borel measure on $[0,1]$ with $\mu((r,1])>0$ for each real $r$ with  $0<r<1$. If $\mu$ has a point mass at 1, then the $L^2(\om)$-norm of an analytic function $f$ is to be understood by $\|f\|^2_{L^2(\om)}=\int_{\Bd}|f|^2d\om + \mu(\{1\})\|f\|^2_{H^2(\dB)}$. We define a norm on $B^N_\om$ by
\begin{align}\label{NormBN} \|f\|^2_{B^N_\om}= \left\{\begin{matrix}& \|f\|^2_{L^2(\om)}, &\text{ if }N=0,\\& \om(\Bd)|f(0)|^2+\|R^N f\|^2_{L^2(\om)}, &\text{ if } N>0,\end{matrix}\right.\end{align}
 and we note that the hypothesis on $\mu$ implies that  each $B^N_\om$ is a Hilbert function space on $\Bd$. For later reference we also note that
 \begin{align}\label{NormNandN-1}\|f\|^2_{B^N_\om}=\om(\Bd)|f(0)|^2+\|Rf\|^2_{B^{N-1}_\om}\end{align} holds for all $N>0$.

An important example is the Drury-Arveson space $H^2_d$. It can be defined as the space of analytic functions $f$ in $\Bd$ such that $$\|f\|^2_{H^2_d} = \sum_{\alpha\in \N_0^d} \frac{\alpha !}{|\alpha|!}|\hat{f}(\alpha)|^2<\infty,$$
where $f$ is given by the power series $f(z)=\sum_{\alpha\in \N_0^d} \hat{f}(\alpha)z^\alpha$ in multinomial notation. One calculates $\|f_n\|^2_{H^2(\dB)}=\frac{n! (d-1)!}{(n+d-1)!}\|f_n\|^2_{H^2_d}$, whenever $f_n$ is a homogeneous polynomial of degree $n$, see e.g. \cite[Section~2]{RiSunkesHankel}. For an arbitrary radially weighted Besov space we therefore have that
\begin{align}\label{sumNorm}\|f\|^2_{B^N_\om}=\om(\Bd) |f(0)|^2 + \sum_{n=1}^\infty n^{2N}\om_n\|f_n\|^2_{H^2_d},\end{align}  where $f=\sum_{n=0}^\infty f_n$ is the representation of $f$ as a sum of homogeneous polynomials of degree $n$, and $\om_n= \frac{n! (d-1)!}{(n+d-1)!}\int_{[0,1]}r^{2n}d\mu(r)$. In particular, the Drury-Arveson space is itself a radially weighted Besov space, up to norm equivalence. In fact,
\begin{align}\label{DruryArveson} H^2_d=\left\{\begin{matrix} B^{(d-1)/2}_\om \ \text{ if } d \text{ is odd and  } \om=\sigma\\B^{d/2}_\om  \ \text{ if } d \text{ is even and } \om=V.\end{matrix}\right.\end{align}
Here $V$ denotes normalized Lebesgue measure on $\Bd$.

The Drury-Arveson space is part of a one-parameter group of spaces of analytic functions. For $\alpha\in \R$ and $f\in \Hol(\Bd)$ define $$\|f\|^2_{D_\alpha(\Bd)}=\sum_{n=0}^\infty (n+1)^\alpha \|f_n\|^2_{H^2_d}\approx \sum_{n=0}^\infty (n+1)^{\alpha+d-1}\int_{\dB}|f_n(z)|^2 d\sigma(z),$$ where $f=\sum_{n=0}^\infty f_n$ is the expansion of $f$ into homogeneous polynomials of degree $n$, and let $D_\alpha(\Bd)=\{f\in \Hol(\Bd): \|f\|^2_{D_\alpha(\Bd)}<\infty\}$. Then $D_0(\Bd)=H^2_d$, $D_{-d+1}(\Bd)=H^2(\dB)$ is the Hardy space, $D_{-d}(\Bd)=L^2_a(\Bd)$ is the Bergman space, and $D=D_1(\D)$ is the Dirichlet space.  Here, and throughout, the equality of spaces is to be understood to include the possibility that the norms are not equal, but equivalent.

More generally, for $\alpha<-d+1$ we have $$(n+1)^{\alpha+d-1}\approx \int_0^1 r^n (1-r)^{-(\alpha+d)}dr.$$ Thus, by setting $$d\om_\alpha(w) =\left\{\begin{matrix} d\sigma(w)\ \ \ \ \ &\text{ if } \alpha=-d+1\\
(1- r)^{-(\alpha+d)}dr d\sigma(z)&\text{ if } w=rz \text{ and }\alpha<-d+1 \end{matrix}\right.$$ we see that $B^0_{\om_\alpha}= D_\alpha(\Bd)$, whenever $\alpha\le -d+1$. This implies that $D_\alpha(\Bd)=B^N_{\om_{\alpha-2N}}$, whenever $N\in \N_0$ with $N\ge \frac{\alpha+d-1}{2}$, cf. \eqref{sumNorm}. If $\alpha >1$, then a simple argument with the Cauchy-Schwarz inequality implies that the spaces $D_\alpha(\Bd)$ are contained in the ball algebra with $\|f\|_\infty \le C\|f\|_{D_\alpha(\Bd)}$, see \cite{Shields} for the case $d=1$. Furthermore, since $f\in D_\alpha(\Bd) \Leftrightarrow Rf\in D_{\alpha-2}(\Bd)$, we conclude that evaluation of $f, Rf, \dots, R^{N-1}f$ at points $z\in \dB$ defines bounded linear functionals on $D_\alpha(\Bd)$ whenever $\alpha > 2N-1$. For more information about these spaces and their multipliers see e.g. \cite{AHMR_RadiallyWeightedBesov, ARS_CarlesonM, BrSh1984, CascanteOrtega2011CarlesonM, CSW_Corona, OrtegaFabrega2006Mult, OrtegaFabrega2000, RiSunkesHankel}, and  Section 14 of \cite{AlHaMcRiFreeOuter}.

If $\HH$ and $\HK$ are Hilbert function  spaces, then $$\mathrm{Mult}(\HH,\HK)=\{\varphi: \varphi f \in \HK \text{ for all }f\in \HH\}$$ are the multipliers from $\HH$ to $\HK$. $\mathrm{Mult}(\HH,\HK)$ is a Banach space with norm $\|\varphi\|_{\mathrm{Mult}(\HH,\HK)}=\sup\{\|\varphi f\|_{\HK}: f\in \HH, \|f\|_{\HH}\le 1\}$. We will write $\mathrm{Mult}(\HH)=\mathrm{Mult}(\HH,\HH)$, and we note that
$$\|\varphi\|_\infty \le \|\varphi\|_{\mathrm{Mult}(B^N_\om)}$$ with equality whenever $N=0$. Furthermore, one  checks that if a function $f$ extends to be analytic in a neighborhood of $\overline{\Bd}$, then $f\in \Mult(B^N_\om)$ for all admissible radial measures $\om$ and all $N\in \N_0$.

We are interested in multiplier invariant subspaces of $\HH$, i.e. those subspaces $\HM \subseteq \HH$ that satisfy $\varphi f \in \HM$, whenever $f\in \HM$ and $\varphi\in \mathrm{Mult}(\HH)$. If $f\in \HH$, then $[f]=\clos_{\HH}\{\varphi f: \varphi \in \mathrm{Mult}(\HH)\}$ denotes the invariant subspace generated by $f$.
A function  $f\in \HH$ is called cyclic in $\HH$ if $[f]=\HH$. If $\HH$ is a radially weighted Besov space, then all polynomials are multipliers, and they are densely contained in $\HH$. Thus, in this case, $f$ is cyclic if and only if $1\in [f]$.

If $\HH=H^2(\D)$, then the cyclic functions are the classical outer functions. If $f\in \HH=H^2_d$, then each slice function $f_z(\lambda)=f(\lambda z)$, $z\in \dB$, is in $H^2(\D)$ with $\|f_z\|_{H^2(\D)}\le \|f\|_{H^2_d}$; see Section \ref{Sec:PickSlices} for a generalization to other spaces.  It follows that if $f$ is cyclic in $H^2_d$, then each slice function $f_z$ must be outer in $H^2(\D)$. However, we will see in Section~\ref{sec:ex} that for $d\ge 2$, there are noncyclic functions $0\ne f\in H^2_d$ such that every slice is outer.

As a tool to investigate the cyclic behaviour of functions in $\HH$, we define the following sets of multipliers, each indexed by an  integer $n\ge 0$:
$$\HC_n(\HH)=\{\varphi\in \MH: \varphi\ne 0 \text{ and } [\varphi^n]=[\varphi^{n+1}]\}.$$
We also define
 $$\HC_\infty(\HH)=\{\varphi\in \MH:\bigcap_{n=1}^\infty [\varphi^n] \ne (0)\}.$$
We consider membership in $\HC_n(\HH)$ to be a weakened form of cyclicity. Indeed, if the multipliers are dense in $\HH$, then $\HC_0(\HH)$ consists of the cyclic multipliers and it is easy to prove that  \begin{align}\label{CnInclusions}\HC_0(\HH)\subseteq \HC_1(\HH) \subseteq \HC_2(\HH) \dots \subseteq \HC_\infty(\HH).
\end{align}

In the case $\HH=H^2(\D)$ one has equality throughout, each set equaling the outer functions in $H^\infty(\D)$, as can be seen from the inner-outer factorization. Similarly, if $\alpha \ge 0$, then we show in Section \ref{Sec:PickSlices} that if $ f\in \HC_n(D_\alpha(\Bd))$ for some $n\in \N$, then for each $z\in \dB$, the slice function $f_z$ is an outer function in $H^\infty(\D)$. One easily checks that for $d=1$ the same conclusion holds for $f\in \HC_\infty(D_\alpha(\Bd))$. However, in Proposition~\ref{SliceFunctionExample} we will present an example to show that if $d\ge 2$, then functions in $\HC_\infty(H^2_d)$ may have slices with non-trivial singular inner factors. It follows that $ \bigcup_{n=0}^\infty \HC_n(H^2_d) \subsetneqq \HC_\infty(H^2_d)$.

For the Dirichlet space $D=D_1(\D)$ it was shown in \cite[Theorem~4.3]{RiSuJOT} that the class $\HC_1(D)$ equals  the outer functions in $\Mult(D)$. On the other hand, it is known \cite{BrSh1984} that there are non-cyclic outer functions in $\Mult(D)$, and therefore $\HC_0(D) \ne \HC_1(D)$. In this particular case we have that  $\HC_1(D)=\HC_\infty(D)$, but nothing like this is true in general.

\begin{example}\label{C2-C1} If $\HH=D_4(\D) = B^2_\om$, $d\om= d\delta_1 \frac{|dz|}{2\pi}$, then $\varphi(z)=1-z$ is an example of an outer function in $\HC_2(B^2_\om)\setminus \HC_1(B^2_\om)$. \end{example}
This holds by observing that the evaluation of functions and their derivatives at $z=1$ define bounded linear functionals on $D_4(\D)$. Hence $1-z \notin [(1-z)^2]$. It is elementary to show that $1-z\in \HC_2(\HH)$, cf. Theorem \ref{PolyCn}.

We finish this section by proving a few more elementary properties of the sets $\HC_\infty(\HH)$.

\begin{lemma}\label{easyMultFree} Assume that $\mathrm{Mult}(\HH) \subseteq \HH$. Then
	
	(a) If $ \varphi\in \HC_\infty(\HH)$, then $\varphi(z) \ne 0$ for all $z\in \Bd$.
	
	(b) If $n,m\in \N_0$ and if $\psi, \varphi \in \mathrm{Mult}(\HH)$ such that $[\varphi^n ]=[\varphi^{n+1}]$ and $[\psi^m]=[\psi^{m+1}]$, then $[\varphi^n \psi^m]=[\varphi^{n+1}\psi^{m+1}]$.
\end{lemma}
\begin{proof} (a) Let $\varphi \in \MH$ and $z_0\in \Bd$ such that $\varphi(z_0)=0$. We need to show that $\bigcap_{n=1}^\infty [\varphi^n]=(0)$, so let $f \in [\varphi^n]$ for each $n\in \N$. Let $r>0$ such that the closure of $B=\{z: |z-z_0|< r\}$ is contained in $\Bd$. It will be sufficient to show that $f(z)=0$ for each $z\in B$. It is clear that $f(z)=0$ at every point $z$ where $\varphi(z)=0$. Let $z\in B$ with $\varphi(z)\ne 0$. For $\lambda\in \C$ set $\psi_z(\lambda)=z_0+\lambda(z-z_0)$. Then $\varphi_z=\varphi\circ \psi_z$ and $f_z=f\circ \psi_z$ are analytic in a neighborhood of the unit disc. Since $f(z)=f_z(1)$ it suffices to show that $f_z=0$.
	
	We have $\varphi(z_0)=0$, but $\varphi(z)\ne 0$. Thus, the function $\varphi_z$ satisfies $\varphi_z(0)=0$, but $\varphi_z$ is not indentically 0. Hence there is $k\ge 1$ and an analytic function $h$ such that $h(0)\ne 0$ and $\varphi_z(\lambda)=\lambda^kh(\lambda)$. Now fix $n \in \N$. Then $f\in [\varphi^n]$ implies that there are multipliers $g_j$ such that $g_j\varphi^n \to f$ in $\HH$. Then $g_j \circ \psi_z \varphi_z^n \to f_z$ uniformly on the closed unit disc. Since each $g_j\varphi^n$ has a zero of multiplicity $nk$ at 0, we conclude that $f_z$ has a zero of multiplicity $nk$ at 0. Since $k\ge 1$ and $n$ is arbitrary, this implies that $f_z=0$.

	(b) Suppose that $n,m, \varphi, \psi$ are as in the hypothesis. Since $\varphi^{n+1}\psi^{m+1}= (\varphi \psi) \varphi^{n}\psi^{m}\in [\varphi^{n}\psi^{m}]$ we only have to show that $\varphi^{n}\psi^{m}\in [\varphi^{n+1}\psi^{m+1}]$. The hypothesis for $\varphi$ implies that there is a sequence of multipliers $u_j$ such that $u_j \varphi^{n+1} \to \varphi^n$ in $\HH$. Then $u_j \varphi^{n+1}\psi^{m+1} \to \varphi^n \psi^{m+1}$ in $\HH$. This implies $\varphi^n \psi^{m+1}\in [\varphi^{n+1}\psi^{m+1}]$. Similarly, there is a sequence of multipliers $v_j$ such that $v_j\psi^{m+1}\to \psi^m$ in $\HH$. Then $\varphi^n v_j\psi^{m+1}\to \varphi^n\psi^m$ and hence $\varphi^n\psi^m \in [\varphi^n \psi^{m+1}]\subseteq [\varphi^{n+1}\psi^{m+1}]$.
\end{proof}

In particular, if $\mathrm{Mult}(\HH)$ is densely contained in $\HH$, then multiplication by cyclic functions preserves each of the classes $\HC_n(\HH)$.

\section{Statements of the main results} \label{sec:main}

In this section we present our main results. Throughout we will suppose that $N\in \N$ and let $\om$ denote an admissible radial measure.

\begin{theorem} \label{AbsoluteValue} If $\varphi, \psi \in \mathrm{Mult}(B^N_\om)$ with $\varphi/\psi \in H^\infty(\Bd)$, then for each $k\in \N$ we have $\varphi^{N+k-1}\in [\psi^k]\subseteq [\psi]$.
\end{theorem}
Consequently, if $N=1$ or $\varphi\in \HC_1(B^N_\om)$, then $|\varphi(z)|\le |\psi(z)|$ implies $\varphi \in [\psi]$. In particular, the cyclicity of $\varphi$ implies cyclicity of $\psi$. A main technical step in the proof of Theorem \ref{AbsoluteValue} is interesting in its own right:

\begin{theorem}\label{thm1} If  $\varphi, \psi \in \mathrm{Mult}(B^N_\omega)$ with $\varphi/\psi \in H^\infty(\Bd)$, then $\varphi^{N+1}/\psi \in \mathrm{Mult}(B^N_\omega)$.
\end{theorem}
Of course, if $N=0$, then $\HH$ is a weighted Bergman space or the Hardy space and $\mathrm{Mult}(B^N_\omega)=H^\infty(\Bd)$. Thus, in this case the Theorem is trivial. However, if $N>0$, then it may happen that $\mathrm{Mult}(B^N_\om) \subsetneq H^\infty(\Bd)$, and hence one may have to choose $n>1$ in order for $\varphi^n/\psi$ to be a multiplier.

By applying Theorem \ref{thm1}  with $\varphi=1$ we recover the following theorem from \cite{LMN20}.
 \begin{corollary} The "one function Corona Theorem" holds for $\mathrm{Mult}(B^N_\om)$. That is, whenever $\psi \in \mathrm{Mult}(B^N_\om)$ with $|\psi(z)|\ge 1$ for all $z\in \Bd$, then $1/\psi \in \mathrm{Mult}(B^N_\om)$. Consequently, for all $\varphi\in \mathrm{Mult}(B^N_\om)$ we have $\sigma(M_\varphi)=\overline{\varphi(\Bd)}$.
 \end{corollary}

Theorem \ref{thm1} is perhaps reminiscent of Wolff's Ideal Theorem for $H^\infty$, \cite{Wolff}. And indeed, as in \cite{BanjadeTrent},  Theorem \ref{thm1} does imply a simple condition for membership in radical ideals generated by principal ideals in $\Mult(B^N_\om)$. For $\psi\in \Mult(B^N_\om)$ let $$\mathrm{Rad}(\psi)=\{\varphi\in \Mult(B^N_\om): \varphi^n=u\psi \text{ for some }u\in \Mult(B^N_\om), \; n\in \N\}.$$

\begin{corollary} If $\varphi, \psi \in \mathrm{Mult}(B^N_\omega)$, then $\varphi\in \mathrm{Rad}(\psi)$ if and only if there is $n\in \N$ such that $\varphi^n/\psi \in H^\infty$.
\end{corollary}

In the course of the proof of Theorem \ref{thm1} we will also establish some uniform norm bounds that are useful for the proof of our next Theorem. Let $\C_{\stable}[z]$ denote the stable polynomials, that is, the polynomials with no zeros in $\Bd$.
\begin{theorem} \label{PolyCn} We have that $\C_{\stable}[z] \subseteq \HC_N(B^N_\om)$. Furthermore, if $N>0$ and $\om$ is of the form $d\om(z)= u(r)2rdr d\sigma(w)$ for some $u\in L^\infty([0,1])$, then $\C_{\stable}[z]\subseteq  \HC_{N-1}(B^N_\om)$.
\end{theorem}
 One may wonder what the smallest $k$ is such that $\C_{\stable}[z] \subseteq \HC_k(B^N_\om)$. Example \ref{C2-C1}  shows that one cannot do any better than Theorem~\ref{PolyCn} for the space $D_4(\D)$.
On the other hand, for each $\om$ there is another admissible radial measure $\om'$ such that $B^N_\om = B^{N+1}_{\om'}$ (with equivalence of norms), see \cite[Theorem~2.4]{AHMR_RadiallyWeightedBesov}. Theorem~\ref{PolyCn} is then of course not sharp for$B^{N+1}_{\om'}$.

For the Drury-Arveson space, represented as a Besov space via \eqref{DruryArveson}, Theorem~\ref{PolyCn} yields that
\begin{align*} \text{ if } d \text{ is odd, then } \C_{\stable}[z]\subseteq \HC_{\frac{d-1}{2}}(H^2_d),\\
\text{ if } d \text{ is even, then } \C_{\stable}[z]\subseteq \HC_{\frac{d}{2}-1}(H^2_d).\end{align*}

In particular, for $d=2$ every stable polynomial is cyclic.
In Section 2 we will show that if $n\le \frac{d}{4}-1$, then $\C_{\stable}[z] \nsubseteq \HC_n(H^2_d)$, see Proposition \ref{EmbedExamples} (b). Thus these inclusions are best possible for $d=2$ and $d=4$, but for other values of $d$ there is potentially a gap.

If $f\in \mathrm{Hol}(\Bd)$ extends to be continuous on $\overline{\Bd}$, then we write $Z(f)=\{z\in \overline{\Bd}:f(z)=0\}$. We say that $f\in \mathrm{Hol}(\Bd)$ satisfies a Lipschitz condition of order $\alpha>0$, if there is $C>0$ such that $|f(z)-f(w)|\le C |z-w|^\alpha$ for all $z,w\in \Bd$. Note that functions that satisfy a Lipschitz condition can be extended to be continuous on $\overline{\Bd}$.
\begin{theorem}\label{ZeroSetThm}
 Let $f,g\in \Mult(B^N_\om)$ be such that
 \begin{enumerate}
 \item $f(z)\ne 0$ and $g(z)\ne 0$ for all $z\in \Bd$,
 \item $f$ extends to be analytic in a neighborhood of $\overline{\Bd}$,
 \item $g$ satisfies a Lipschitz condition of order $\alpha>0$.
 \end{enumerate}
Assume that $Z(f)\cap \dB \subseteq Z(g)\cap \dB$. Then there is an $n\in \N$ such that $g^n\in [f]$. Furthermore, if we additionally assume that $g$ is a polynomial, then $g^N\in [f^N]\subseteq [f]$ for every $N \geq 1$.
\end{theorem}
In particular, if $g$ is cyclic, then $f$ is cyclic. Thus, for polynomials  $p\in \C_{\stable}[z]$, the geometry of $Z(p)\cap \dB$ determines whether or not $p$ is cyclic. If $w\in \partial \Bd$, then $p(z)=1-\la z, w\ra$ is a polynomial such that $Z(p)\cap \dB=\{w\}$ and it is easily seen that $p$ is cyclic in $H^2_d$. This implies that if $f$ extends to be analytic in a neighborhood of $\overline{\Bd}$, has no zeros in $\Bd$, and only finitely many zeros in $\dB$, then $f$ is cyclic in $H^2_d$. In Examples \ref{Exa:2dim} and \ref{Exa:2dim2} we will give examples of cyclic polynomials in $H^2_d$ such that $Z(p)\cap \dB$ has 1 or 2 real dimensions. However, in Theorem \ref{thm:Noncyclic} we will show that if $Z(p)\cap \dB$ embeds a cube of real dimension $\ge 3$, then $p$ is not cyclic in $H^2_d$.

If the radially weighted Besov space happens to be what is called a complete Pick space (see Section \ref{SecPickSpaces} for the definitions), then the results of Theorem \ref{AbsoluteValue} can be partially extended to apply to arbitrary functions in $B^N_\om$. The spaces $D_\alpha(\Bd)$ are complete Pick spaces for all $\alpha \ge 0$. This includes the Dirichlet space $D$ and the Drury-Arveson space $H^2_d$. Furthermore, in \cite{AHMR_RadiallyWeightedBesov} general conditions were given on $\om$ that imply that $B^N_\om$ is a complete Pick space. For example, if  $\alpha >-1$ and if $d\om(z) = w(z) dV(z)$ where $\frac{w(z)}{(1-|z|^2)^\alpha}$  is non-decreasing as $|z|\to 1$,
then for $N\ge \frac{\alpha+d}{2}$  the space $B_\omega^N$ is a complete Pick space.
For us complete Pick spaces are important, because if $\HH$ is a complete Pick space, and if $f\in \HH$, then $f=u/v$ for some $u, v\in \MH$ with $v$ cyclic. \begin{theorem} \label{AbsoluteValuePickIntro} Let $N\in\N$, and let $B^N_\om$ be a radially weighted Besov space that is also a complete Pick space.
Let $f,g\in B^N_\om$ be such that $f/g\in H^\infty(\Bd)$.

 If $N=1$ or if $f=u/v$ for $v$ cyclic and  $u\in \HC_1(B^N_\om)$, then $f\in [g]$.
\end{theorem}

 It will follow that if $|f(z)|\le |g(z)|$, and if $f$ is cyclic, then $g$ is cyclic. Since the constant 1 is cyclic in $B^N_\om$ the Theorem implies that any $g\in \HH$ that is bounded below, must be cyclic. Thus, Theorem \ref{AbsoluteValuePickIntro}  improves Theorem 1.5 of \cite{RiSunkesHankel}, where the Theorem was proved only for $H^2_d$ under the additional assumptions that $f=1$ and  $g$ be in the Bloch space.

For the Dirichlet space $D$ of the unit disc Theorem \ref{AbsoluteValuePickIntro} was known, see Corollary 5.5 of \cite{RiSuMMJ}. The proof here is considerably less technical than the one in  \cite{RiSuMMJ}. Theorem \ref{AbsoluteValuePickIntro} will follow from Theorem \ref{AbsoluteValuePick}, which contains a slightly more general result.


\section{Some Examples} \label{sec:ex}
In order to illustrate our theorems we start with some examples for  $H^2_d$. We present two ways to embed $D_{\frac{k-1}{2}}(\D)$ in $H^2_d$, where $1\le k\le d$. The first one of these is  well-known, and has been used before to construct functions with interesting properties in the Drury-Arveson space, see e.g. \cite[Theorem 3.3]{Arveson},  \cite[Lemma 2.1]{HartzHenkin},  \cite[Lemma 9.1]{AHMRSubhomogeneous}, or see \cite[Example 2]{BenetCondoriEtAl} for a bidisc version of such an embedding.

 Note that if $z=(z_1,\dots, z_d)\in \Bd$, then  the geometric-arithmetic mean inequality implies that for each integer $k$ with $1\le k\le d$ we have  $$\left(\prod_{j=1}^k|z_j|^2\right)^{1/k}\le \frac{\sum_{j=1}^k|z_j|^2}{k}\le \frac{1}{k}.$$
 Hence $\tau_k(z)=k^{k/2}\prod_{j=1}^kz_j$ maps $\Bd$ into $\D$.

\begin{lemma}\label{Lem:DiriEmbed} Let $1\le k\le d$. The operator $$T_{k,d}: D_{\frac{k-1}{2}}(\D)\to H^2_d,\ \  T_{k,d}f=f\circ \tau_k$$ is bounded and bounded below. Furthermore, if $\varphi \in \Mult(D_{\frac{k-1}{2}}(\D))$, then  $T_{k,d}\varphi \in \Mult(H^2_d)$.
\end{lemma}
If $d=k$, then this is the special case of $s=1$ of \cite{AHMRSubhomogeneous}, Lemma 9.1 and Proposition 9.4. If $d>k$, then we combine this with use of the isometric embedding of $H^2_k$ in $H^2_d$ given by $f\to f\circ P$, where $P$ is the projection from $\C^d\to \C^k, (z_1,\dots, z_d)\to (z_1, \dots, z_k)$. This embedding is also isometric as a map between multiplier algebras $\Mult(H^2_k)\to \Mult(H^2_d)$, see  Lemma 6.2 of \cite{AHMRSubhomogeneous}.

\begin{lemma}\label{Lem:Tkd} Let $1\le k\le d$.

(a) If $f\in D_{\frac{k-1}{2}}(\D)$, then $f$ is cyclic in $D_{\frac{k-1}{2}}(\D)$, if and only if $T_{k,d}f$ is cyclic in $H^2_d$.

(b)  Let $n\in \N$ and $f\in \Mult(D_{\frac{k-1}{2}}(\D))$, then  $f\in  \HC_n(D_{\frac{k-1}{2}}(\D))$,if and only if $T_{k,d}f\in \HC_n(H^2_d)$.

(c) If $f\in D_{\frac{k-1}{2}}(\D)$ is an outer function, then for each $z\in \dB$ the slice function $(T_{k,d}f)_z$ is outer. Here  $(T_{k,d}f)_z(\lambda)=T_{k,d}f(\lambda z)$, $\lambda\in \D$.
\end{lemma}
\begin{proof} We prove (b) and (c). The proof of (a) is similar to (b).

(b) Fix $n\in \N$ and  $f\in \Mult(D_{\frac{k-1}{2}}(\D))$. Then by Lemma \ref{Lem:DiriEmbed} $T_{k,d}f\in \Mult(H^2_d)$.
If   $f\in  \HC_n(D_{\frac{k-1}{2}}(\D))$, then there is a sequence of polynomials $\{p_j\}$ such that $p_jf^{n+1}\to f^n$ in $D_{\frac{k-1}{2}}(\D)$. Then  for each $j$ we have  $q_j=T_{k,d}p_j$ is a polynomial and by  Lemma \ref{Lem:DiriEmbed} $(T_{k,d}p_j) (T_{k,d}f)^{n+1}=T_{k,d}(p_jf^{n+1})\to T_{k,d}(f^n)=(T_{k,d}f)^n$ in $H^2_d$. Hence $T_{k,d}f\in \HC_n(H^2_d)$.

Conversely, if $T_{k,d}f\in \HC_n(H^2_d)$, then there are polynomials $q_j\in \Mult(H^2_d)$ such that $q_j (T_{k,d}f)^{n+1} \to (T_{k,d}f)^n$ in $H^2_d$.

For $n\in \N_0$ define $\alpha_n=(n,\dots,n,0,\dots,0)\in \N_0^d,$ where the first $k$ components of $\alpha_n$ equal $n$ and the remaining components are 0.
If $q_j(z)=\sum_{\alpha}\hat{q}_j(\alpha)z^{\alpha}$, then let $P_kq_j(z)=\sum_{n\ge 0}\hat{q}_j(\alpha_n) \prod_{k=1}^d z_k^n$. Note that $P_kq_j=T_{k,d}p_j$ for some polynomial $p_j$ and that
$$(q_j-P_kq_j)T_{k,d}(f^{n+1})\perp (P_kq_j) T_{k,d}(f^{n+1})-T_{k,d}(f^n)$$ by the orthogonality of the monomials in $H^2_d$. Hence
\begin{align*}
\|p_jf^{n+1}-f^n\|^2_{D_{(k-1)/2}} &\approx \|T_{k,d}(p_jf^{n+1})-T_{k,d}(f^n)\|^2_{H^2_d}\\
&=\|(P_kq_j) T_{k,d}(f^{n+1})-T_{k,d}(f^n)\|^2_{H^2_d}\\
&\le \|(q_j-P_kq_j)T_{k,d}(f^{n+1})\|^2 + \|(P_kq_j) T_{k,d}(f^{n+1})-T_{k,d}(f^n)\|^2_{H^2_d}\\
&= \|q_j (T_{k,d}f)^{n+1}-(T_{k,d}f)^n\|^2_{H^2_d} \to 0
\end{align*}
as $j\to \infty$.  Thus, $f\in \HC_n(D_{(k-1)/2})$.

(c) If  $f\in D_{\frac{k-1}{2}}(\D)$ is an outer function, then $\log |f(0)|= \int_0^{2\pi} \log|f(e^{it})|\frac{dt}{2\pi}$. Let $z\in \dB$, then $(T_{k,d}f)_z(e^{it})= f(e^{ikt} \tau_k(z))$ and hence
\begin{align*}\int_0^{2\pi} \log|(T_{k,d}f)_z(e^{it})|\frac{dt}{2\pi}&= \int_0^{2\pi} \log|f(e^{ikt}\tau_k(z))|\frac{dt}{2\pi}\\&=\log|f(0)|=\log|(T_{k,d}f)_z(0)|.\end{align*}
Hence $(T_{k,d}f)_z$ is outer.
\end{proof}
\begin{example} \label{Exa:2dim} If $1\le k\le 3$, $f(z)=1-z$, then $T_{k,d}f$ is cyclic in $H^2_d$. Furthermore, the set $Z(T_{k,d}f)\cap \dB$ equals an embedded $k-1$ dimensional cube.
\end{example}
It is well-known that $f$ is cyclic in $D_\alpha(\D)$ for $\alpha \le 1$, see \cite{BrSh1984}. It also follows from the second part of Theorem \ref{PolyCn}, which will be proved later. Hence the cyclicity of $T_{k,d}f$ follows from Lemma \ref{Lem:Tkd} (a). The statement about the zero set is also easily seen. For example, if $k=3$, then $$Z(T_{k,d}f)\cap \dB=\{ 3^{-1/2}(e^{it},e^{is},e^{-i(t+s)},0,\dots, 0): t,s\in [0,2\pi]\}.$$

Similarly one sees that for $k\ge 4$ the set $Z(T_{k,d}f)\cap \dB$ embeds a cube of dimension $k-1\ge 3$, hence it will follow from Theorem \ref{thm:Noncyclic} that $T_{k,d}f$ is not cyclic for any $k\ge 4$. Alternatively, that  will also follow from part (b) of the following proposition.
\begin{prop}\label{EmbedExamples} (a) If $d\ge 2$, then there is non-cyclic $f\in H^2_d$ such that every slice function $f_z$ is outer in $H^2(\D)$.

(b) If $d\ge k\ge 4n>0$ and $p(z)=1-z$, then the  polynomial $T_{k,d} p\notin  \HC_{n-1}(H^2_d)$. Hence $\C_{\stable}[z] \nsubseteq \HC_{n-1}(H^2_d)$.
\end{prop}

\begin{proof} (a)  By Lemma \ref{Lem:Tkd} it will be enough to show that there is a noncyclic outer function $f\in D_{1/2}(\D)$, because then $T_{2,d}f$ will be the required example. That is known, we quickly give an overview of the idea. The proof uses Carleson sets and $\alpha$-capacity for $\alpha=1/2$. Indeed, using Theorem 3 of Section IV of \cite{CarlesonBook} one constructs a "generalized Cantor" set $E$ with positive $1/2$-capacity (also see \cite{ElFKRans}, Section 4). Generalized Cantor sets are Carleson sets. Thus, by results of Carleson  for any $n\in \N$ there is an outer function $f\in C^n(\overline{\D})\cap \Hol(\D)$ such that $f=0$ on $E$ (\cite{CarlesonUniquenessSets}). Then $f$ is not cyclic in $D_{1/2}$, see e.g. \cite{ElFKRans} Theorem 1.1.

(b) If  $d\ge k\ge 4n>0$, then $\frac{k-1}{2}>2n-1$. Then, as noted in the Introduction, it follows that $f\to f^{(j)}(z)$ defines a bounded linear functional on $D_{(k-1)/2}(\D)$ for all $z\in \partial \D$ and all $j=0, 1, \dots, n-1$. But then the functional of evaluation of the $n-1$-derivative at 1 annihilates every function in $[(1-z)^n]$, but it does not annihilate $(1-z)^{n-1}$. Hence $(1-z)^{n-1}\notin [(1-z)^n]$. This implies that $1-z\notin \HC_{n-1}(D_{(k-1)/2}(\D))$. Then part (b) of the proposition follows from Lemma \ref{Lem:Tkd}.
\end{proof}

The second way to embed $D_{(k-1)/2}$ in $H^2_d$ is given by the following lemma.

\begin{lemma}\label{Lem:2ndDiriEmbedding}Let $1\le k\le d$. Then the operator $S_k: D_{\frac{k-1}{2}}(\D)\to H^2_d, S_kf(z)=f(\sum_{j=1}^k z_j^2)$ is bounded and bounded below. \end{lemma}

\begin{proof} As above, using the embeeding $H^2_k\subseteq H^2_d$ it suffices to prove the case $k=d$.
 If $f(\lambda)=\sum_{n=0}^\infty a_n\lambda^n$, then $$S_df(z)=\sum_{n=0}^\infty a_n (\sum_{j=1}^d z_j^2)^n=\sum_{n=0}^\infty a_n \sum_{|\alpha|=n}\frac{|\alpha|!}{\alpha!}z^{(2\alpha)}.$$ Then
$$\|S_df\|^2_{H^2_d}=\sum_{n=0}^\infty |a_n|^2 \frac{(n!)^2}{(2n)!} \sum_{|\alpha|=n} \frac{(2\alpha)!}{(\alpha!)^2 }.$$ Thus we have to prove that $$(n+1)^{(d-1)/2} \approx \frac{(n!)^2}{(2n)!} \sum_{|\alpha|=n} \frac{(2\alpha)!}{(\alpha!)^2 },$$ where the implied constants may depend on $d$, but not on $n$. For $n=0$ we have equality, so it will be enough to consider $n\ge 1$.
We will prove the statement by induction on $d$. For $d=1$ the statement holds with equality, and we will also explicitly verify the case $d=2$. Note by Stirling's formula we have that $(n!)^2/(2n)! \approx \sqrt{n}/2^{2n}$ for $n \geq 1$. Then
\begin{align*}\frac{(n!)^2}{(2n)!} \sum_{|\alpha|=n} \frac{(2\alpha)!}{(\alpha!)^2 }&=\frac{(n!)^2}{(2n)!} \sum_{k=0}^n \frac{(2k)!}{(k!)^2}\frac{(2(n-k))!}{ ((n-k)!)^2 }\\
&\approx 2+ \frac{\sqrt{n}}{2^{2n}}\sum_{k=1}^{n-1} \frac{(2k)!}{(k!)^2}\frac{(2(n-k))!}{ ((n-k)!)^2 }\\
&\approx 2+ \sqrt{n} \sum_{k=1}^{n-1} \frac{1}{\sqrt{k}\sqrt{n-k}}\\
 & \approx \sqrt{n}\sum_{1 \leq k \leq n/2} \frac{1}{\sqrt{k} \sqrt{n-k}} \\
&\approx \sum_{1 \leq k \leq n/2} \frac{1}{\sqrt{k}} \approx \sqrt{n}.
\end{align*}

Now assume that $d\ge 2$ and the statement holds for $d$. We will show that it also holds for $d+1$. Note that for each $n\ge 0$ we have
$$\{\beta\in \N_0^{d+1}: |\beta|=n\}=\{(\alpha,n-|\alpha|): \alpha\in  \N_0^{d}, 0\le |\alpha|\le n\}$$ and hence by the induction hypothesis
\begin{align*}
\frac{(n!)^2}{(2n)!} \sum_{\beta\in N_0^{d+1},|\beta|=n} \frac{(2\beta)!}{(\beta!)^2 }&=\frac{(n!)^2}{(2n)!}\sum_{k=0}^n \sum_{\alpha\in \N_0^d,|\alpha|=k} \frac{(2\alpha)!(2(n-k))!}{(\alpha!)^2 ((n-k)!)^2 } \\
&\approx \frac{(n!)^2}{(2n)!}\sum_{k=0}^n\frac{(2k)!}{(k!)^2}\frac{(2(n-k))!}{ ((n-k)!)^2 }(k+1)^{(d-1)/2}\\
&\le (n+1)^{(d-1)/2} \frac{(n!)^2}{(2n)!}\sum_{k=0}^n\frac{(2k)!}{(k!)^2}\frac{(2(n-k))!}{ ((n-k)!)^2 }\\
&\approx (n+1)^{(d-1)/2} \sqrt{n}\le   (n+1)^{d/2} \ \ \text{by the $d=2$ case.}
\end{align*}
 Thus, we have the required upper bound. For the lower bound note that by symmetry
\begin{align*}2\sum_{k=0}^n\frac{(2k)!}{(k!)^2}\frac{(2(n-k))!}{ ((n-k)!)^2 }&(k+1)^{(d-1)/2}\\
&=\sum_{k=0}^n\frac{(2k)!}{(k!)^2}\frac{(2(n-k))!}{ ((n-k)!)^2 }((k+1)^{(d-1)/2}+(n-k+1)^{(d-1)/2})\\
&\ge \sum_{k=0}^n\frac{(2k)!}{(k!)^2}\frac{(2(n-k))!}{ ((n-k)!)^2 } (\frac{n}{2}+1)^{(d-1)/2},
\end{align*}
and now we can substitute this into the previous formula and obtain the lower bound with a similar calculation as before.
\end{proof}

 \begin{example}\label{Exa:2dim2} If $d\ge 3$, then $p(z)=1-(z_1^2+z_2^2+z_3^2)$ is cyclic in $H^2_d$ and $Z(p)\cap \dB$ is a  2-dimensional cube.\end{example} Again we note  that  for $k=3$ we have $D_{(k-1)/2}(\D)=D$ is the classical Dirichlet space and the polynomial $1-z$ is cyclic in $D$. Hence the  statement follows from Lemma \ref{Lem:2ndDiriEmbedding} with arguments that are analogous to the proof of Lemma \ref{Lem:Tkd} (a) by use of Lemma \ref{Lem:DiriEmbed}. In this case
 $$Z(p)\cap \dB=\{(\cos t \cos s,\cos t \sin s, \sin t,0, \dots, 0): t,s\in [0,2\pi]\}.$$

To complete this section, we will follow the ideas of Brown and Shields \cite{BrSh1984} (see also \cite{AlanSola} for a $\mathbb{B}_2$-version) to obtain a necessary condition for cyclicity in $H^2_d$, proving that the zero set of a cyclic function $f \in H^2_d \cap C(\overline{\mathbb{B}_d})$ cannot embed a 3-dimensional cube. To prove this we want to construct bounded linear functionals of the form
 $$f\to\int_{\partial\mathbb{B}_d} fd\mu,\quad  f\in H_d^2\cap C(\overline{\mathbb{B}_d}),$$
 for appropriate finite Borel measures $\mu$ on $\partial\mathbb{B}_d$, such that these functionals annihilate nontrivial  multiplier-invariant subspaces. The argument applies directly to  Hilbert function spaces whose kernel has the form $(1-\langle z,w\rangle)^{-\alpha}$, but here we shall focus on $H_d^2$.
 \begin{lemma}\label{large-zero-set} Let $\mu$ be a finite Borel measure on  $\partial\mathbb{B}_d$ and
 	$$f_\mu(z)=\int\frac1{1-\langle z,w\rangle} d\mu(w),\quad z\in \mathbb{B}_d.$$
 	(i) If
 	\begin{equation}\label{energy_condition} E(\mu)=\int\int\frac1{|1-\langle z,w\rangle|} d\mu(w)d\mu(z)<\infty,\end{equation}
 	then $f_\mu\in H_d^2$ and 	$\|f_\mu\|_{H_d^2}^2\le E(\mu)$.\\
 	(ii) If $E(\mu)<\infty$ and $f\in H_d^2\cap C(\overline{\mathbb{B}_d})$, then
 	$$\langle f,f_\mu\rangle_{H_d^2}=\int fd\mu.$$
 \end{lemma}
 \begin{proof} (i) For fixed $r\in (0,1)$, the $H_d^2-$ valued function $u_r(z)= k_{rz}$ is continuous on the closed unit ball  $\overline{\mathbb{B}_d}$. For a measure $\mu$ as in the statement, consider the Bochner integral $\int u_rd\mu$. Evaluating at $z\in \mathbb{B}_d$  and using  elementary properties of Bochner integrals yields
 	$$\left(\int u_rd\mu\right)(z)=\int \langle u_r, k_z\rangle_{H_d^2} d\mu= 	f_\mu(rz)= (f_\mu)_r(z).$$
 	But then the same properties of the Bochner integral yield for $f\in H_d^2$
 	\begin{equation}\label{energy_equation}\langle 	f,(f_\mu)_r\rangle_{H_d^2} =\int f_rd\mu.\end{equation}
 	In particular,	for $f=(f_\mu)_r$ we obtain
 	$$\|(f_\mu)_r\|_{H_d^2}^2 =\int\int \frac1{1-r^2\langle z, w\rangle }d\mu(w)d\mu(z)\le \int\int \frac1{|1-r^2\langle z, w\rangle |}d\mu(w)d\mu(z).$$
 	Now let $r\to 1$,  use the inequality  $$|1-r^2\zeta|\ge \frac1{2}|1-\zeta|,$$
 	together with the dominated convergence theorem to conclude that
 	$$\limsup_{r\to 1}\|(f_\mu)_r\|_{H_d^2}^2\le E(\mu).$$
 	Since $(f_\mu)_r(z)\to f_\mu(z), ~z\in \mathbb{B}_d$, it follows that
 	$f_\mu\in H_d^2$ and $\|f_\mu\|_{H_d^2}^2\le E(\mu)$, which proves (i). Part (ii) follows from  (i) and  \eqref{energy_equation}, since
 	$$\int f_rd\mu=\langle 	f,(f_\mu)_r\rangle =\langle 	f_r, f_\mu\rangle.$$
 	If $f\in H_d^2\cap C(\overline{\mathbb{B}_d})$ we obtain the result letting  $r\to 1$.
 \end{proof}
Given a set $S\subseteq\mathbb{C}^d$ and an integer $m\ge 0$, we  say that $S$ contains  an embedded cube of dimension $m$ if  there exists a diffeomorphism $\phi$ from $(-1,1)^m$ into $S$.
 \begin{theorem}\label{thm:Noncyclic} Let $f\in H_d^2\cap C(\overline{\mathbb{B}_d})$ and assume that $Z(f)\cap\partial\mathbb{B}_d$ contains an embedded cube of dimension $m\ge3$. Then $f$ is not cyclic in $H_d^2$.	
 \end{theorem}

 \begin{proof} Let $\phi :(-1,1)^m\to U\subseteq Z(f)\cap\partial\mathbb{B}_d$ be a diffeomorphism. Then $\phi$ satisfies for some $c>0$ that
 	\begin{equation}\label{reverse-Lipschitz}|\phi(t)-\phi(s)|\ge c|t-s|,\quad t,s\in (-1,1)^m.\end{equation}
 	Consider the pushforward measure
 	$$\mu(E)= \lambda_m(\phi^{-1}(E\cap U))$$
 	on $\partial \mathbb{B}_d$, where $\lambda_m$ denotes the $m-$dimensional Lebesgue measure on $(-1,1)^m$.  According to Lemma \ref{large-zero-set} it will be sufficient to show that \eqref{energy_equation} holds. Indeed, in this case part (ii) of the lemma gives that $f_\mu\in [f]^\perp$, and clearly $f_\mu\ne 0$. To demonstrate this, we observe that
 	$$|1-\langle z,w\rangle|\ge\text{Re }(1-\langle z,w\rangle)=\frac{|z-w|^2}{2}, \qquad z ,w \in \partial \mathbb{B}_d,$$
 	to obtain that
 	$$E(\mu)\le \frac2{c}\int_{(-1,1)^m}\int_{(-1,1)^m}
 	\frac1{|t-s|^2} d\lambda_m(s)d\lambda_m(t)<\infty,$$
 	since $m\ge 3$. \end{proof}


\section{Radially weighted Besov spaces} \label{sec:besov}
Throughout this section, $\om$ will denote an admissible radial measure, see Section~\ref{sec:prelim}.
\subsection{Lemmas about ratios of multipliers}
We start with a lemma, which is basically from \cite{AHMR_RadiallyWeightedBesov}, and which says that all radially weighted Besov spaces satisfy the "multiplier inclusion condition" with constant 1.
\begin{lemma} \label{MultIncl} For each $k\in \N$ the space $\mathrm{Mult}(B^{k}_\om)$ is contractively contained in $\mathrm{Mult}(B^{k-1}_\om)$, that is, $$\|\varphi\|_{\mathrm{Mult}(B^{k-1}_\om)}\le \|\varphi\|_{\mathrm{Mult}(B^{k}_\om)} \ \text{  for all }\varphi\in \mathrm{Mult}(B^{k}_\om).$$
\end{lemma}
\begin{proof} If the measure $d\om(z) = d\mu(r) d\sigma(w)$ is such that $\mu$ is absolutely continuous, then this follows directly from Theorem 1.2 or Corollary 3.4 of \cite{AHMR_RadiallyWeightedBesov}. But the proof given in \cite{AHMR_RadiallyWeightedBesov} actually applies to the more general situation considered here, where $\mu$ is not assumed to be absolutely continuous. Indeed, it follows from (\ref{sumNorm}) that the reproducing kernel $K^k$ for $B^k_\om$ is given by $$K^k_w(z)= \frac{1}{\om(\Bd)}+ \sum_{n=1}^\infty \frac{1}{n^{2k}\om_n}\la z,w\ra^n.$$ Thus, we can use proposition 3.3 of \cite{AHMR_RadiallyWeightedBesov} with the sequences $\{a_n\}$ and $\{b_n\}$ where $a_0=b_0=\frac{1}{\om(\Bd)}$, $a_n= \frac{1}{n^{2k}\om_n}$ and $b_n=\frac{1}{n^{2(k-1)}\om_n}$ for $n \ge 1$.
\end{proof}

For us, the multiplier inclusion condition is important because the following lemma is now elementary.

\begin{lemma}\label{basicMultIncl}
We have that  $\varphi \in \mathrm{Mult}(B^k_\om)$ if and only if $\varphi \in \mathrm{Mult}(B^{k-1}_\om)$ and $R\varphi \in \mathrm{Mult}(B^k_\om, B^{k-1}_\om)$. Furthermore,
\begin{align} \label{3.1}&\|\varphi\|_{\mathrm{Mult}(B^k_\om)}\le 2(\|\varphi\|_{\mathrm{Mult}(B^{k-1}_\om)} + \|R\varphi\|_{\mathrm{Mult}(B^k_\om,B^{k-1}_\om)} ) \text{ and }\\\label{3.2}  &\|R\varphi\|_{\mathrm{Mult}(B^k_\om,B^{k-1}_\om)} \le 2 \|\varphi\|_{\mathrm{Mult}(B^k_\om)}\end{align}
\end{lemma}
\begin{proof} Since $f\in B^k_\om$ if and only if $Rf\in B^{k-1}_\om$, we see that $\varphi\in \mathrm{Mult}(B^k_\om)$, if and only if $(R\varphi)f+\varphi Rf\in B^{k-1}_\om $ for each $f\in B^k_\om.$ Thus, if $\varphi \in \mathrm{Mult}(B^{k-1}_\om)$ and $R\varphi \in \mathrm{Mult}(B^k_\om, B^{k-1}_\om)$, then
$\varphi\in \mathrm{Mult}(B^k_\om)$. Conversely, if $\varphi\in \mathrm{Mult}(B^k_\om)$, then by the multiplier inclusion condition $\varphi\in \mathrm{Mult}(B^{k-1}_\om)$, and hence the identity $(R\varphi)f= R(\varphi f) -\varphi Rf$ implies that $R\varphi \in \mathrm{Mult}(B^k_\om, B^{k-1}_\om)$.
That argument can be used to get the estimates.
Let $f\in B^k_\om$. Then by equation (\ref{NormNandN-1}) we have
\begin{align*} \|\varphi f\|^2_{B^k_\om} &=\om(\Bd)|(\varphi f)(0)|^2+\|(R\varphi)f+\varphi Rf\|^2_{B^{k-1}_\om}\\
&\le \|\varphi\|^2_{\mathrm{Mult}(B^{k-1}_\om)}\om(\Bd)|f(0)|^2 \\ & \ \ \ \  + 2(\|R\varphi\|^2_{\mathrm{Mult}(B^k_\om, B^{k-1}_\om)} \|f\|^2_{B^k_\om} +\|\varphi\|^2_{\mathrm{Mult}(B^{k-1}_\om)}\|Rf\|^2_{B^{k-1}_\om)})\\
&\le 2 (\|R\varphi\|^2_{\mathrm{Mult}(B^k_\om, B^{k-1}_\om)}+\|\varphi\|^2_{\mathrm{Mult}(B^{k-1}_\om)})\|f\|^2_{B^k_\om}
\end{align*}
This proves (\ref{3.1}). Furthermore, since all functions in $(R\varphi)f= R(\varphi f) -\varphi Rf$ are 0 at the origin we have
\begin{align*}\|(R\varphi)f\|_{B^{k-1}_\om}&\le \|R(\varphi f)\|_{B^{k-1}_\om}+ \|\varphi Rf\|_{B^{k-1}_\om}\\
&\le \|\varphi f\|_{B^{k}_\om} + \|\varphi\|_{\mathrm{Mult}(B^{k-1}_\om)} \|Rf\|_{B^{k-1}_\om}\\
&\le ( \|\varphi\|_{\mathrm{Mult}(B^k_\om)}+\|\varphi\|_{\mathrm{Mult}(B^{k-1}_\om)})\|f\|_{B^{k}_\om}\\
&\le 2 \|\varphi\|_{\mathrm{Mult}(B^k_\om)}\|f\|_{B^{k}_\om},\end{align*} where the last inequality followed from Lemma \ref{MultIncl}.
Hence (\ref{3.2}) holds.  \end{proof}

The following lemma is key to proving Theorem~\ref{AbsoluteValue}. Note that it immediately implies Theorem \ref{thm1}.
\begin{lemma}\label{LemRatioEstimate}
If  $M>0$ and $\varphi, \psi \in \mathrm{Mult}(B^N_\om)$ with \begin{enumerate}
\item $\|\varphi\|_{\mathrm{Mult}(B^N_\om)}, \|\psi\|_{\mathrm{Mult}(B^N_\om)} \le M$ and
 \item $\frac{\varphi}{\psi}\in H^\infty(\Bd)$ with $\|\frac{\varphi}{\psi}\|_\infty\le 1$,\end{enumerate} then for all $ s\in \N$ and integers $k$ with $ 0\le  k\le N$ we have $\frac{\varphi^{s+k}}{\psi^s} \in \mathrm{Mult}(B^k_\om)$ with
\begin{equation}\label{RatioMult}\left\|\frac{\varphi^{s+k}}{\psi^s}\right\|_{\mathrm{Mult}(B^k_\om)}\le 8^k (s+k)^kM^k
\end{equation}
Furthermore, if the functions $\varphi, \psi$ are nonzero in $\Bd$, then the conclusion and inequality (\ref{RatioMult}) hold for all real $s >0$.
\end{lemma}

\begin{proof} Note that  if $s$ is a positive integer or if the functions $\varphi, \psi$ are nonzero in $\Bd$ and $s\in (0,\infty)$, then $\varphi^s/\psi^s \in H^\infty(\Bd)$.  This is the only place that the different hypotheses on $s$ and $\varphi, \psi$ are used, and in the following we will treat these cases simultaneously.

We start by noting that Lemma \ref{MultIncl}  implies that $\|\varphi\|_{\mathrm{Mult}(B^k_\om)} \le M$ and $ \|\psi\|_{\mathrm{Mult}(B^k_\om)} \le M$ for all $k$ with $0\le k\le N$ and hence by Lemma  \ref{basicMultIncl} we have $\|R\varphi\|_{\mathrm{Mult}(B^k_\om, B^{k-1}_\om)} \le 2M$ and $\|R\psi\|_{\mathrm{Mult}(B^k_\om, B^{k-1}_\om)}\le 2M$, if $1\le k\le N$.

We will now establish the lemma by induction on $k$.
Since $\mathrm{Mult}(B^0_\om)=H^\infty(\Bd)$ it is clear that the case $k=0$ holds. Now suppose that $1\le k \le N$ and the inequality (\ref{RatioMult}) holds for $k-1$ for all $s >0$ and for some  $\varphi, \psi$ that satisfy (i) and (ii). Then in particular for a fixed $s$ it also holds for $s+1$, that is,
\begin{align*}\left\|\frac{\varphi^{s+k-1}}{\psi^s}\right\|_{\mathrm{Mult}(B^{k-1}_\om)}&\le 8^{k-1} (s+k-1)^{k-1}M^{k-1}\le 8^{k-1} (s+k)^{k-1}M^{k-1}\\
\left\|\frac{\varphi^{s+k}}{\psi^{s+1}}\right\|_{\mathrm{Mult}(B^{k-1}_\om)}&\le 8^{k-1} (s+k)^{k-1}M^{k-1}.
\end{align*}
We compute
\begin{align*}R \frac{\varphi^{s+k}}{\psi^s}= (s+k)\frac{\varphi^{s+k-1}}{\psi^s}R\varphi - s \frac{\varphi^{s+k}}{\psi^{s+1}} R\psi, \end{align*} and hence

\begin{align*}\left\|R \frac{\varphi^{s+k}}{\psi^s}\right\|_{\mathrm{Mult}(B^k_\om,B^{k-1}_\om)} &\le (s+k)\left\|\frac{\varphi^{s+k-1}}{\psi^s}\right\|_{\mathrm{Mult}(B^{k-1}_\om)} \|R\varphi\|_{\mathrm{Mult}(B^k_\om,B^{k-1}_\om)} \\ & \ \  + s \left\|\frac{\varphi^{s+k}}{\psi^{s+1}}\right\|_{\mathrm{Mult}(B^{k-1}_\om)} \|R\psi\|_{\mathrm{Mult}(B^k_\om,B^{k-1}_\om)} \\
&\le  (2s+k)8^{k-1}(s+k)^{k-1}M^{k-1}(2M)\\
&=2(2s+k)8^{k-1}(s+k)^{k-1}M^k\end{align*}
by the induction hypothesis as stated above. Hence by inequality (\ref{3.1}) and the induction hypothesis we obtain
\begin{align*}\left\|\frac{\varphi^{s+k}}{\psi^s}\right\|_{B^k_\om}&\le 2\left(\left\|\varphi\frac{\varphi^{s+k-1}}{\psi^s}\right\|_{B^{k-1}_\om}+2(2s+k)8^{k-1}(s+k)^{k-1}M^k\right)\\
&\le 2 (M 8^{k-1}(s+k)^{k-1}M^{k-1}+2(2s+k)8^{k-1}(s+k)^{k-1}M^k)\\
&\le 2(1+2(2s+k))8^{k-1}(s+k)^{k-1}M^k\\
&\le 8^{k}(s+k)^{k}M^k.
\end{align*}
\end{proof}

\subsection{The proof of Theorem \ref{PolyCn}}

If $f\in \mathrm{Hol}(\Bd)$ and $0<r<1$, we write $f_r(z)=f(rz)$.
\begin{theorem}\label{f&f_r} Let $N\in \N_0$.
 If $M>0$ and $\varphi\in \mathrm{Mult}(B^N_\om)$ such that $\varphi(z)\ne 0$ for each $z\in \Bd$ and $\|\frac{\varphi}{\varphi_r}\|_\infty\le M$ for all $0<r<1$, then $[\varphi^N]=[\varphi^{N+1}]$, that is, $\varphi\in \HC_N(B^N_\om)$. \end{theorem}
\begin{proof}
By hypothesis $\frac{1}{\varphi_r}\in \Mult(B^N_\om)$, and thus $\frac{\varphi^{N+1}}{\varphi_r}\in [\varphi^{N+1}]$ for each $r$.
By Lemma \ref{LemRatioEstimate} with $\psi= C\varphi_r$, $k=N$, and $s=1$ there is a constant $K>0$ such that $\|\frac{\varphi^{N+1}}{\varphi_r}\|_{\Mult(B^N_\om)}\le K$ for all $0<r<1$. This implies that $\frac{\varphi^{N+1}}{\varphi_r} \to \varphi^N$ weakly in $B^N_\om$, and thus that $\varphi^N\in [\varphi^{N+1}].$
 \end{proof}
Theorem~\ref{f&f_r} applies to polynomials in $\C_{\stable}[z]$, and thus implies the first part of Theorem \ref{PolyCn}. In order to prove the remaining part of Theorem \ref{PolyCn} we need estimates for $p^n/p_r$ and its derivatives for $n \ge 1$.  We will prove the needed results by looking at slice functions. Thus, we start with a single variable lemma (see Lemma 1 of \cite{KneseKosinskiRansfordSola} for the case $n=1$).
\begin{theorem} Let $n,m\in \N$. There is a constant $c=c(n,m)$ such that whenever  $p$ is a polynomial of degree $\le m$ that has no zeros in $\D$, then for all $0\le r<1$ we have
	\begin{align*}\left| \frac{d^k}{dz^k} \frac{p(z)^n}{p(rz)}\right|\le c &\ |p(0)|^{n-1}  \ \text{ for all } 0\le k<n, z\in \D, \ \text{ and }\\
		\int_{|z|<1} \left| \frac{d^n}{dz^n} \frac{p(z)^n}{p(rz)}\right|^2 &\frac{dA(z)}{\pi} \le c\ |p(0)|^{2n-2} .\end{align*}
\end{theorem}
\begin{proof} The statement is obviously true for constant polynomials, thus we may assume that the degree of $p$ equals $m\ge 1$. By dividing through by $p(0)$ we may also assume that $p(0)=1$. Then there are $A_1, \dots A_m\in \overline{\D}$ such that $$\frac{p(z)^n}{p(rz)}= \prod_{j=1}^m \frac{(1-A_jz)^n}{(1-A_jrz)} \ \ z\in \D, 0\le r\le 1.$$
	
	Write $g_{A,r}(z)=\frac{(1-Az)^n}{(1-Arz)}$, then by the multi-product Leibniz formula we have for each $0\le k\le n$
	$$\frac{d^k}{dz^k} \frac{p(z)^n}{p(rz)}= \sum_{|\alpha|=k}\frac{|\alpha|!}{\alpha !} \prod_{j=1}^m g_{A_j,r}^{(\alpha_j)}(z).$$

	The Theorem will follow, if we show that there is $c>0$ that is independent of $A\in \overline{\D}$ and  $0 \le r\le 1$ such that  $|g_{A,r}^{(j)}(z)|\le c$, whenever $0\le j<n$  and $z\in \D$, and $\int_{|z|<1} |g_{A,r}^{(n)}(z)|^2 \frac{dA(z)}{\pi} \le c$.
	
	\
	
	Note that  $|g^{(j)}_{A,r}(z)|=|g^{(j)}_{1,r}(Az)A^j|\le |g^{(j)}_{1,r}(Az)|$. Hence it suffices to prove the statement for $A=1$. If $0\le r\le 1/2$, then the function $\frac{(1-z)^n}{1-rz}$ and all of its derivatives are rational functions with poles in $\{|z|\ge 2\}$ that are continuous as  functions of the parameter $r$, hence by compactness there is $c>0$ such that $|g^{(j)}_{1,r}(z)|\le c$ for all $0\le j\le n$, $ |z|\le 1$ and all $ 0\le r\le 1/2.$

	Now let $ 1/2\le r < 1$. For $z\in \D$ set $w=1-rz$. Then \begin{align*}g_{1,r}(z)&=\frac{((r-1)+w)^n}{r^nw}\\&=\frac{1}{r^n} \sum_{k=0}^n \binoN(r-1)^{n-k}w^{k-1}\\&=\frac{1}{r^n} \left(\frac{(r-1)^n}{1-rz}+ q(r,z)\right)\end{align*}
	where $q(r,z)$ is a polynomial expression in the variables $r$ and $z$.

	Let $0\le j\le n$ and take the $j$-th derivative
	
	\begin{align*} |g^{(j)}_{1,r}(z)|&=r^{-n}\left|\frac{j!r^j(r-1)^n}{(1-rz)^{j+1}})+\frac{\partial^j}{\partial z^j}q(r,z)\right|\\
		&\le 2^{-n} \left( \frac{j!(1-r)^n}{|1-rz|^{j+1}}+|\frac{\partial^j}{\partial z^j}q(r,z)|\right)
	\end{align*}
	By compactness the expressions $|\frac{\partial^j}{\partial z^j}q(r,z)|\le C$, where $C$ is independent of $0\le j\le n$, $|z|\le 1$ and $1/2\le r\le 1$.
	Furthermore, if $j<n$, then $ \frac{j!(1-r)^n}{|1-rz|^{j+1}}\le (n-1)!$, hence the boundedness statement follows in those cases. Finally, we have that
	$$\int_\D |g^{(n)}_{1,r}(z)|^2 \frac{dA(z)}{\pi}\le 2 \cdot 2^{-2n}\left(\int_{\D}\frac{(n!)^2(1-r)^{2n}}{|1-rz|^{2n+2}}\frac{dA(z)}{\pi}+ C^2\right)\le C(n)$$
	see for example \cite[Theorem~1.7]{HedKorZhu}.
\end{proof}
We can now prove the second part of Theorem \ref{PolyCn}. For convenience we restate it here.
\begin{theorem} Let $N\in \N$ and let $\om$ be an admissible radial measure of the type $d\om(w)=u(r)2rdr d\sigma(z)$, $u\in L^\infty(0,1)$. Then every polynomial $p$ without zeros in $\Bd$ satisfies $[p^{N-1}]=[p^N]$ in $B^N_\om$, that is, $\C_{\stable}[z]\subseteq \HC_{N-1}(B^N_\om)$.\end{theorem}

\begin{proof} Let $f\in \Hol(\Bd)$, and let $f=\sum_{n=0}^\infty f_n$ be the representation as sum of homogeneous polynomials of degree $n$.  For $\lambda\in \C$ let $f_z(\lambda)=f(\lambda z)$ be a slice function at $z\in \dB$.
Write $D_\lambda =\lambda \frac{\partial}{\partial \lambda}$, then for $f=\sum_{n=0}^\infty f_n$ we have $Rf(\lambda z)=\sum_{n=1}^\infty n f_n(z) \lambda^n= D_\lambda f_z(\lambda)$. Then $$R^Nf(\lambda z)= \sum_{k=1}^N a_k \lambda^k \frac{\partial^k}{\partial \lambda^k}f_z(\lambda)$$ for some coefficients $a_1, \dots, a_N$, and
\begin{align*} \int_{\Bd}|R^N f|^2 d\om&=
 \int_{\dB}\int_{[0,1]}\int_0^{2\pi}|R^N f(re^{it}z)|^2 \frac{dt}{2\pi} w(r)2rdrdtd\sigma(z)\\
&\le \|w\|_\infty \int_{\dB}\int_\D\left |\sum_{k=1}^N a_k \lambda^k \frac{\partial^k}{\partial \lambda^k}f_z(\lambda)\right|^2 \frac{dA(\lambda)}{\pi} d\sigma(z)\\
&\le N\|w\|_\infty\sum_{k=1}^N |a_k| \int_{\dB}\int_\D \left| \frac{\partial^k}{\partial \lambda^k}f_z(\lambda)\right|^2 \frac{dA(\lambda)}{\pi} d\sigma(z).\end{align*}
Hence, by the lemma,
\begin{align*}\int_{\Bd} \left|R^N \frac{p^N}{p_r}\right|^2 d\om&\le N \|w\|_\infty\sum_{k=1}^N |a_k|\ c |p(0)|^{2N-2}
\end{align*}
This implies that $\frac{p^N}{p_r}\to p^{N-1}$ weakly in $B^N_\om$ as $r\to 1$, and therefore that $p^{N-1}\in [p^N]$.
\end{proof}

\subsection{The proof of Theorem \ref{AbsoluteValue}}

In Lemma \ref{LemRatioEstimate} we proved that for $\varphi, \psi\in \Mult(B^N_\om)$ with $\varphi/\psi\in H^\infty$ we have $\varphi^{N+k}/\psi^k \in \Mult(B^N_\om)$. That immediately implies that $\varphi^{N+k}\in [\psi^k]$. We will now take advantage of the fact that $\om$ is a radial measure to show that actually $\varphi^{N+k-1}\in [\psi^k]$. That is, we will prove Theorem \ref{AbsoluteValue}. We start with some preliminaries.

\begin{lemma} Let $N\ge 1$.For any $f\in B^N_\om$ we have
$$\|f-f_r\|_{B^{N-1}_\om} \le (1-r) \|f\|_{B^{N}_\om} , 0<r<1.$$
\end{lemma}
\begin{proof} Let $f= \sum_{n=0}^\infty f_n$ be the expansion of $f$ into a sum of homogeneous polynomials of degree $n$. Then for $0<r<1$ we have
\begin{align*}\|f-f_r\|^2_{B^{N-1}_\om}&= \sum_{n=0}^\infty \|(1-r^n)f_n\|^2_{B^{N-1}_\om}\\  &\le (1-r)^2 \sum_{n=0}^\infty n^2 \|f_n\|^2_{B^{N-1}_\om}\\ &\le(1-r)^2\|f\|^2_{B^{N}_\om}.\end{align*} \end{proof}
\begin{lemma} Let $k\ge 0$ and $\varphi \in \mathrm{Mult}(B^k_\om)$. Then
$$\|R\varphi_r\|_{\mathrm{Mult}(B^k_\om)} \le \frac{\|\varphi\|_{\mathrm{Mult}(B^k_\om)}}{1-r^2}.$$
\end{lemma}

\begin{proof} For $\lambda\in \D, z\in \Bd$ define $F(\lambda)(z)=\varphi(\lambda z)$. Then $F:\D \to \Mult(B^k_\om)$ is an analytic function with $\|F(\lambda)\|_{\Mult(B^k_\om)}\le \|\varphi\|_{\Mult(B^k_\om)}$ for all $\lambda \in \D$. Then by the Cauchy formula we obtain $\|F'(\lambda)\|_{\Mult(B^k_\om)}\le \frac{\|\varphi\|}{1-|\lambda|^2}$. The Lemma follows, because $(R\varphi_r)(z)= rF'(r)(z)$.
\end{proof}

\begin{lemma}\label{LemRadialEstimate} If $N\ge 1$ and if $f,g \in B^N_\om$ such that $\varphi=f/g \in \mathrm{Mult}(B^{N-1}_\om)$, then for all $0\le r<1$ we have
$$\|\varphi_r g\|_{B^N_\om}\le 3\|\varphi\|_{\mathrm{Mult}(B^{N-1}_\om)} \|g\|_{B^N_\om}+\|f\|_{B^N_\om}$$ and hence
$f\in [g]$, since  $\varphi_rg\to f$ weakly in $B^N_\om$.
\end{lemma}
\begin{proof} We have $$\|\varphi_r g\|_{B^N_\om}\le\|\varphi_r (g-g_r)\|_{B^N_\om}+\|\varphi_r g_r\|_{B^N_\om} \le\|\varphi_r (g-g_r)\|_{B^N_\om}+\|f\|_{B^N_\om}.$$ Since $g-g_r$ vanishes at the origin we have $\|\varphi_r (g-g_r)\|_{B^N_\om}= \|R(\varphi_r (g-g_r))\|_{B^{N-1}_\om}$. Using the previous two lemmas, we have
\begin{multline*} \|R(\varphi_r (g-g_r)) \|_{B^{N-1}_\om} \le \|(g-g_r)R\varphi_r\|_{B^{N-1}_\om}+\|\varphi_r R(g-g_r)\|_{B^{N-1}_\om}\\
\le \frac{\|\varphi\|_{\mathrm{Mult}(B^{N-1}_\om)}}{1-r^2} (1-r) \|g\|_{B^{N}_\om} + \|\varphi\|_{\mathrm{Mult}(B^{N-1}_\om)} \|R(g-g_r)\|_{B^{N-1}_\om}\\
\le 3 \|\varphi\|_{\mathrm{Mult}(B^{N-1}_\om)} \|g\|_{B^{N}_\om},
\end{multline*}
concluding the proof.
\end{proof}
\begin{proof}[Proof of Theorem \ref{AbsoluteValue}] Let $N\in \N$ and $\varphi, \psi \in \Mult(B^N_\om)$ with $\varphi/\psi \in H^\infty$. Let $k\in \N$, and set $f= \varphi^{N+k-1}$ and $g=\psi^k$. Then by Lemma \ref{LemRatioEstimate} we have $\frac{f}{g}=\frac{\varphi^{N+k-1}}{\psi^k}\in \Mult(B^{N-1}_\om)$. Hence Lemma \ref{LemRadialEstimate} implies $\varphi^{N+k-1}\in [\psi^k]$.
\end{proof}


\subsection{Cyclicity and zero sets}
If $f:U\to \R$ is a function, then let  $Z(f)=\{x\in U:f(x)=0\}$ be the zero locus of $f$. As in \cite{KneseKosinskiRansfordSola} we will use the {\L}ojasiewicz inequality from real algebraic geometry, see \cite{Loja}.
\begin{lemma} Let $U \subseteq \mathbb{R}^{2d}$ be open, and let $f \colon U \to \mathbb{R}$ be a real analytic function such that $Z(f)\ne \emptyset$. Then for every compact set $K \subseteq U$ there are positive constants $p$ and $C$ such that
$$\dist (x, Z(f))^p \leq C|f(x)| \ \text{ for all }x\in K.$$
\end{lemma}
We obtain the following.
\begin{lemma} \label{LojaCoro}
Suppose that $f, g\in \Hol(\Bd)$ such that
\begin{enumerate}
\item $f(z)\ne 0, g(z)\ne 0$ for all $z\in \Bd$,
\item $f$ extends to be analytic in a neighborhood of $\overline{\Bd}$,
\item $g$ satisfies a Lipschitz condition of order $\alpha >0$.
\end{enumerate}
Let $Z(g) \subseteq \partial \Bd$ denote the zero set of the Lipschitz extension of $g$.

If $Z(f) \cap \partial \Bd \subseteq Z(g)\cap \dB$, then there is a constant $C > 0$ and an integer $j > 0$ such that
$${|g(z)|^j} \leq C{|f(z)|}$$ for all $ z \in \Bd.$
\end{lemma}
\begin{proof} If $Z(f) \cap \partial \Bd=\emptyset$, then $|f|$ is bounded below on $\overline{\Bd}$ and the conclusion of the lemma follows with $j=1$. Thus,  we assume $Z(f) \cap \partial \Bd\ne\emptyset$.
Since  $f$ has no zeroes in $\Bd$ we have
$$1-\sum_{i=1}^d |z_i|^2\le 2 \dist(z, Z(f)\cap \partial \Bd) \text{ for all } z\in \Bd.$$

Then by the {\L}ojasiewicz inequality applied with $K=\overline{\Bd}$ there is an even integer $n$ and a $C_1>0$ such that
$$\left(1- \sum_{i=1}^d |z_i|^2 \right)^{n} \le 2^n \dist(z, Z(f))^{n} \le C_1 |f(z)|^2, \quad z \in \Bd.$$
Next we apply the {\L}ojasiewicz inequality to the function
$$r(z) = |f(z)|^2 + \left(1- \sum_{i=1}^d |z_i|^2  \right)^{n},$$ and we find that there is an integer $m$ and $C_2>0$ such that for all $z\in \Bd$
$$\dist(z, Z(f) \cap \partial \Bd)^m \le C_2 r(z) \le C_2(1+C_1)|f(z)|^2 .$$
On the other hand, by the Lipschitz property of $g$, we have that there is $C_3>0$ such that
$$|g(z)| \le C_3 \dist(z, Z(g))^\alpha \leq C_3\dist(z, Z(f) \cap \partial \Bd)^\alpha \text{ for all }z\in \Bd.$$
This proves the lemma with $j \ge  \frac{m}{2\alpha}$.
\end{proof}
\begin{proof}[Proof of Theorem \ref{ZeroSetThm}] Let $f,g \in \Mult(B^N_\om)$ as in the hypothesis of the theorem. Then by Lemma \ref{LojaCoro} there is $j\in \N$ such that $\frac{g^j}{f}\in H^\infty$. Then Theorem \ref{AbsoluteValue} implies that for each $k\in \N$ we have $g^{j(N+k-1)}\in [f^k]\subseteq [f]$. This proves the first part of the theorem with $n=jN$. If $g$ is a polynomial, then  $[g^N]=[g^m]$ for all $m\ge N$. Hence taking $k=N$ we obtain $g^N\in [g^N]=[g^{j(2N-1)}]\subseteq [f^N]$.
\end{proof}

\section{Complete Pick spaces}\label{SecPickSpaces}

\subsection{Background}
Recall that each Hilbert function space $\HH$ has a reproducing kernel $k: X \times X\to \C$. Writing $k_w(z)=k(z,w)$ it  satisfies $f(w)=\la f, k_w\ra$ for all $f\in \HH$, $w\in X$.  A  reproducing kernel $k$ on $X$ is called a normalized complete Pick kernel, if there is $w_0\in X$ and a function $b$ from $X$ into some auxiliary Hilbert space $\HK$ such that $b(w_0)=0$ and
$$k_w(z)= \frac{1}{1-\la b(z),b(w)\ra_\HK}.$$ In the interesting case where $\HH$ is a Hilbert space of analytic functions one easily shows that $\HH$ is separable, and then one may assume that $\HK$ is separable also.

 We say that a Hilbert function space $\HH$ on $X$ is a complete Pick space, if there is a Hilbert space norm on $\HH$ that is equivalent to the original norm, and such that the reproducing kernel that $\HH$ has with respect to the new norm is a normalized complete Pick kernel. The Hardy space $H^2$ of the unit disc is the easiest example of a complete Pick space. The Drury-Arveson space $H^2_d$ is an example of a complete Pick space on $\Bd$, with reproducing kernel $k_w(z)=\frac{1}{1-\la z,w\ra}$.
Superharmonically weighted Dirichlet spaces on the unit disc are further examples of complete Pick spaces \cite{ShimorinCNP}.  The results of \cite{AHMR_RadiallyWeightedBesov} yield further large classes of radially weighted Besov spaces that are also complete Pick spaces.

We now recall some important basic properties of complete Pick spaces.
\begin{lemma} \label{kernelMult} If $k$ is a normalized complete Pick kernel and if $w\in X$, then $k_w\in \rm{Mult}(\HH)$ with $\|k_w\|_{\rm{Mult}(\HH)} \le 2 \|k_w\|^2_{\HH}.$
\end{lemma}
For spaces of analytic functions this was proved in \cite{GrRiSu}. For the general version see \cite[Proposition 4.4]{Serra} or \cite[Lemma~7.2]{AlHaMcRiFreeOuter}.
\begin{corollary}If $\HH$ is a complete Pick space, then $\rm{Mult}(\HH)$ is dense in $\HH$. In particular, $f\in \HH$ is cyclic if and only if $1\in [f]$.\end{corollary}

We shall need another useful estimate in spaces with  a normalized complete Pick kernel. For spaces of analytic functions this was also proved in \cite{GrRiSu}, while the general version given here follows from Corollary 3.3 in \cite{AHMRFactor}.

\begin{lemma} \label{pointwiseestimate} If $k$ is a normalized complete Pick kernel, $f\in \HH$, and $w\in X$, then $$|f(w)|^2\le 2\text{\rm Re} \langle f, k_w f\rangle_\HH -\|f\|_{\HH}^2.$$
\end{lemma}

Finally we recall a special case of \cite[Theorem 1.1 (i)]{AHMRFactor}.
 \begin{theorem} \label{SarasonAnalog} Let $\HH$ be a complete Pick space with $k_{w_0}=1$. For  $f : X\to \C,$  the following are equivalent:
  \begin{enumerate}
  \item $f\in \HH$  and $ \|f\|\le 1 $;
\item there are multipliers $\varphi, \psi\in \text{\rm Mult}(\HH)$ such that \begin{enumerate}
\item $f= \frac{\varphi}{1-\psi}$
\item $\psi(w_0)=0$, and
\item $\|\psi h\|^2+\|\varphi h\|^2 \le \|h\|^2 $ for every $h\in \HH$.
\end{enumerate}
\end{enumerate}
\end{theorem}
A direct application we will use frequently is that each complete Pick space $\HH$ is contained in the corresponding Pick-Smirnov class
 $$N^+(\HH)=\{\varphi/\psi : \varphi, \psi \in \mathrm{Mult}(\HH), \psi \text{ cyclic}\},$$ see Lemma~\ref{cyclicSubspace} below.

\subsection{Cyclic subspaces} The aim of this subsection is to show that for radial Besov spaces which are also  complete Pick spaces there is variant of Theorem \ref{AbsoluteValue} which refers to functions in $\HH$ rather than multipliers.
We begin by listing three observations regarding subspaces of the form $[f]$.

  \begin{lemma} \label{cyclicSubspace} Let $\HH$ be a separable Hilbert function space on $X$. If $f= \varphi/(1-\psi)\in \HH$, where  $\varphi, \psi \in \rm{Mult}(\HH)$ and $ \psi\ne 1$, $\|\psi\|_{\rm{Mult}(\HH)}\le 1$, then $1-\psi$ is cyclic in $\HH$ and $[f]=[\varphi]$.
   \end{lemma}
\begin{proof}
This is a  is a straightforward combination of Lemma 2.3 of \cite{AHMRsccps} and Lemma 3.6 (a) of \cite{AHMRwp}.
\end{proof}

\begin{lemma} \label{multIndep} Let $\HH$ be a separable Hilbert function space on $X$. If  $f=\frac{u}{v}= \frac{u_1}{v_1}\in N^+(\HH)$, where $u, v, u_1, v_1\in \rm{Mult}(\HH)$, $v, v_1$ cyclic, then  $[u^n]=[u_1^n]$, for all $n\in  \mathbb{N}$.
\end{lemma}
\begin{proof}  We have $u^nv_1^n=v^nu^n_1$, hence $u^nv^n_1 \in [v^nu^n_1]\subseteq [u^n_1]$. Since $v_1$ is cyclic, so is $v_1^n$, and  since $u^n$ is a multiplier,  it easily follows that $u^n\in[u^n_1]$. By symmetry $u^n_1\in [u^n]$, hence $[u^n]=[u^n_1]$.
\end{proof}

  \begin{lemma} \label{multIndep1}  Let $\HH$ be a complete Pick space.
   If $f=\frac{u}{v}\in \HH$, where $u, v\in \rm{Mult}(\HH)$ and $v$ is cyclic, then $[u]=[f]$.
  \end{lemma}
  \begin{proof} Let $f=\frac{u}{v}$. By Theorem \ref{SarasonAnalog} there are $\varphi, \psi\in \text{\rm Mult}(\HH)$ such that $\|\psi\|_{\rm{Mult}(\HH)}\le 1$, $\psi \ne 1$, and
  $f= \frac{ \varphi}{1-\psi}$. By Lemma \ref{cyclicSubspace} we have $[f]=[\varphi]$, hence we have to show $[u]=[\varphi]$ and that follows from Lemma \ref{multIndep}.
  \end{proof}

With these lemmas in hand we can turn to the main result of this subsection, which contains Theorem~\ref{AbsoluteValuePickIntro} as a special case.

\begin{theorem} \label{AbsoluteValuePick} Let $N\in\N$, and let $B^N_\om$ be a radially weighted Besov space that is also a complete Pick space. Let $f,g\in B^N_\om$ be such that $f/g\in H^\infty$.

If  $f=\frac{u}{v}$, where $u, v\in \mathrm{Mult}(B^N_\om)$, $v$ is cyclic, then $u^{N-1}f\in [g]$.	If $u\in \HC_n(B^N_\om)$ for some $1\le n\le N$, then $u^{n-1}f\in [g]$.
\end{theorem}
 \begin{proof} Write $f=\frac{u}{v},~g= \frac{a}{b}$ with $a,b,u,v\in \text{\rm{Mult}}(B_\om^N)$, $v$ and $b$ cyclic.
 	By Lemma \ref{multIndep1} we have that $[f] = [u]$ and $[g]=[a]$.  	
 	Theorem \ref{AbsoluteValue} with $k=1$ gives us that $b^N u^N \in [va] = [a]$, and thus that $u^N \in [a] = [g]$. Now Lemma \ref{multIndep1} implies $u^{N-1}f=\frac{u^N}{v}\in [u^N]\subseteq [g]$.

 If $u \in \HC_n(B^N_\om)$, then $u^{n}\in [u^{n+k}]$ for all nonnegative integers $k$, and hence  $u^n\in [u^N]\subseteq [g]$. As above this and Lemma~\ref{multIndep1} implies that $u^{n-1}f \in [g]$.
 \end{proof}

\subsection{Inner factors of slices}\label{Sec:PickSlices} As usual, for a function $u:\mathbb{B}_d\to\mathbb{C}$ and $z\in \partial\mathbb{B}_d$ the corresponding slice function $u_z:\mathbb{D}\to\mathbb{C}$ is given by $$u_z(\lambda)=u(\lambda z),\quad \lambda\in \mathbb{D}.$$
The following simple observation follows directly from Lemma \ref{pointwiseestimate}.
\begin{lemma} \label{sliceest} Let $N\in\N$, and let $B^N_\om$ be a radially weighted Besov space that is also a complete Pick space. Then there is $c>0$ such that whenever  $f\in B^N_\om$, then every  slice $f_z \in H^2(\D)$ with $$\|f_z\|_{H^2(\D)}\le c\|f\|_{B^N_\om}, \quad z\in \partial\mathbb{B}_d.$$	
\end{lemma}
\begin{proof} Let $f\in B^N_\om,~z\in \partial\mathbb{B}_d$. By Lemma \ref{pointwiseestimate} we have for all $\lambda\in \mathbb{D}$
	$$|f_z(\lambda)|^2\le  2\text{\rm Re} \langle f, k_{\lambda w} f\rangle -\|f\|^2,$$
	for a suitable equivalent norm on $B^N_\om$ and the induced scalar product. The right-hand side is harmonic in $\lambda$ and its value at 0 equals   $\|f\|^2$.
 Thus $|f_z|^2$ has a harmonic majorant in the unit disc, that is, $f_z\in H^2(\D)$ with $\|f_z\|^2_{H^2(\D)}\le \|f\|^2\le c\|f\|^2_{B^N_\om}$, where the constant $c$ appears, because of the equivalence of norms.
\end{proof}
An application of the lemma yields the following result.
\begin{prop} \label{slice-outer} Let $N\in\N$, and let $B^N_\om$ be a radially weighted Besov space that is also a complete Pick space. If $\varphi\in \HC_n(B_\om^N)$ for some $n\ge 0$, then every slice  $\varphi_z,~z\in \partial\mathbb{B}_d$, is an outer function in $H^\infty$.	
\end{prop}
\begin{proof} By assumption we have that there is a sequence $(u_k)$ in $\text{\rm Mult}(B_\om^N)$ such that $(u_k\varphi^{n+1})$ converges to $\varphi^n$ in $B^N_\om$. Then Lemma  \ref{pointwiseestimate}
	implies that for every $z\in \partial\mathbb{B}_d$, $(u_k)_z\varphi^{n+1}_z\to \varphi_z^n$ in $H^2(\D)$. Since  $(u_k)_z\in H^\infty$, this   shows that $\varphi_z\in C_n(H^2(\D))$, i.e. it is  outer.	
\end{proof}
It is natural to ask whether this result continues to hold under the weaker assumption that the multiplier $\varphi$ belongs to $\HC_\infty(B_\om^N)$? Of course the answer is affirmative in the one variable case. However, it turns out that this is no longer the case when $d>1$. Our counterexample is based on the following result about the structure of the Drury-Arveson space $H_d^2$. It is a precise version of the decomposition used in the remarks after Theorem 4.3. of \cite{GrRiSu}. For $n\ge 0$ let $\HK_n$ be the space of analytic functions on the unit disc with reproducing kernel $k_\lambda(z)=(1-\overline{\lambda}z)^{-n-1}$.
\begin{lemma}   Let   $d\in \N,~d>1$. For  $\alpha\in \mathbb{N}_0^{d-1}$, $z=(z_1,\ldots,z_{d-1})\in \mathbb{C}^{d-1}$, let $e_\alpha(z_1,\ldots, z_{d-1})=\sqrt{\frac{\alpha!}{|\alpha|!}} z^\alpha$. Then the map from $\HK_{|\alpha|}$ to $H_d^2$,
	$$f\to g,\quad g(z)=e_\alpha(z_1,\ldots,z_{d-1})f(z_d),$$
	is an isometry, and
		$$H_d^2=\bigoplus_{\alpha\in \mathbb{N}_0^{d-1}}e_\alpha\HK_{|\alpha|}.$$
\end{lemma}
\begin{proof} A direct computation reveals that for all $k\in \mathbb{N}$, we have  $$\|w^k\|_{\HK_{|\alpha|}}=\|e_\alpha(z_1,\ldots,z_{d-1})z_d^k\|_{H_d^2}$$
	and since normalized monomials form an orthonormal  basis in $\HK_{|\alpha|}$, the first assertion follows. It is also clear that the subspaces $e_\alpha\HK_{|\alpha|}\subseteq H_d^2$ are pairwise orthogonal and since their sum contains all monomials it must equal the whole space $H_d^2$.	
\end{proof}
It is important to note that $\HK_0=H^2$, while for $|\alpha|>0$ $\HK_{|\alpha|}$ is a weighted Bergman space with $\HK_1=L_a^2$, the unweighted Bergman space on the unit disc. In particular, if $u\in H^\infty$ then the function $$v(z_1,\ldots,z_d)=u(z_d),\quad (z_1,\ldots,z_d)\in \mathbb{B}_d,$$
is a multiplier of $H_d^2$ with $$\|v\|_{\text{\rm
		Mult}(H_d^2)}=\|u\|_\infty.$$
	Finally, for the result below we shall use the well known fact that there exist singular inner functions $\theta\in H^\infty$ such that $\theta$ is cyclic in each of the spaces $\HK_{|\alpha|}$, $|\alpha|>0$ (see \cite{HedKorZhu}).
	\begin{prop}\label{SliceFunctionExample} Let $\theta\in H^\infty$ be singular inner  such that $\theta$ is cyclic in each of the spaces $\HK_{|\alpha|}$ for $|\alpha|>0$, and for $d>1$ set
	 $$\varphi(z_1,\ldots,z_d)=\theta(z_d),\quad (z_1,\ldots,z_d)\in \mathbb{B}_d,$$	
Then \begin{equation}\label{slice-c-infty} [\varphi^n]=\varphi^ne_{(0,\ldots,0)}\HK_0\oplus(H_d^2\ominus e_{(0,\ldots,0)}\HK_0).\end{equation}
In particular, $\varphi\in \HC_\infty(H_d^2)$ but for $z=(0,\ldots,0,1)\in \partial\mathbb{B}_d$ we have that $\varphi_z(\lambda)=\theta(\lambda)$ is an inner function.		
		\end{prop}
\begin{proof}
The statement is self-explanatory, since $\varphi^ne_\alpha\HK_{|\alpha|}$  is contained and dense in $e_\alpha\HK_{|\alpha|}$ when $|\alpha|>0$, which immediately leads to \eqref{slice-c-infty}. Clearly,
$$\bigcap_{n\ge 1}[\varphi^n] =H_d^2\ominus e_{(0,\ldots,0)}\HK_0,$$
i.e. $\varphi\in \HC_\infty(H_d^2)$.
\end{proof}


\section{Further questions} \label{sec:questions}

We start with the obvious question that we have left open.

\begin{ques} If $d\in \N$, then what is the smallest $n$ such that $\C_{\stable}[z]\subseteq \HC_n(H^2_d)$?\end{ques}

We think of functions in the classes $\HC_n(\HH)$ as $H^2_d$-analogues of functions without inner factors. With this in mind we formulate a weakened form of the Brown-Shields conjecture for $H^2_d$, see \cite{BrSh1984}.
\begin{ques} Given $d\in \N$,  is there is $N\in \N$ such that whenever $f\in \HC_n(H^2_d)$ for some $n$, then  $f\in \HC_N(H^2_d)$?\end{ques}
There are natural related questions.
\begin{ques} If $f\in \Mult(H^2_d)$ such that every slice $f_z$ is outer, then is $f\in \HC_N(H^2_d)$ for some $N$?\end{ques}
We mentioned in the introduction that the analogous questions for the Dirichlet space $D$ have a positive answer. We finish by providing some further evidence that this might extend to $H^2_d$. Let $A^\infty(\D)=\{f\in C^\infty(\overline{\D}): f|\D \text{ analytic}\}$. Then $A^\infty(\D)\subseteq \Mult(D_{\alpha})$ for all $\alpha\in \R$.

\begin{prop}If $f\in A^\infty(\D)$ is outer and if $1\le k \le d$, then $T_{k,d}f\in \HC_{n}(H^2_d)$ for every $n\ge \frac{k-1}{4}$.
\end{prop}
\begin{proof}
Let $1\le k \le d$, $n\ge \frac{k-1}{4}$, and let $f\in A^\infty(\D)$ be outer. By Lemma \ref{Lem:Tkd} it suffices to show that $f\in \HC_n(D_{(k-1)/2}(\D))$. The choice of $n$ implies that  $D_{2n}(\D)\subseteq D_{(k-1)/2}$ with $\|g\|_{D_{(k-1)/2}}\le \|g\|_{D_{2n}(\D)}$ for all $g\in D_{2n}(\D)$. Hence it will be enough to show that $f\in \HC_n(D_{2n}(\D))$. For the spaces $D_{2n}(\D)$ Korenblum determined the invariant subspaces, \cite{KorenblumH2n}. Indeed, each  invariant subspace that contains an outer function is of the form
$$I(E_0,E_1, \dots, E_{n-1})=\{f\in D_{2n}(\D): f^{(j)}(z)=0 \text{ on }E_j \text{ for }j=0,\dots n-1\},$$
where    $\partial \D \supseteq E_0\supseteq E_1\supseteq \dots \supseteq E_{n-1}$ are compact sets such that $E_0$ is a Carleson set and  $E_0\setminus E_{n-1}$ is discrete. If $E=\{z\in \partial \D: f(z)=0\}$, then it follows that $[f^n]=[f^{n+1}]=I(E,\dots,E)$. Hence $f\in \HC_n(D_{2n}(\D))$.\end{proof}

\bibliography{CyclicBib}

\end{document}